\documentclass[a4paper, 11pt, twoside]{article}
\usepackage{amsthm}
\usepackage{amsmath}
\usepackage{fancyhdr}
\usepackage{amssymb}
\usepackage{mathtools}
\usepackage{indentfirst}
\usepackage{xcolor}
\usepackage{bm}
\usepackage{extarrows}
\usepackage[pagewise]{lineno}

\theoremstyle{plain}
\newtheorem{proposition}{Proposition}
\newtheorem{lemma}[proposition]{Lemma}
\newtheorem{theorem}[proposition]{Theorem}

\theoremstyle{definition}

\newtheorem{remark}[proposition]{Remark}

\DeclareMathOperator*{\esssup}{ess\,sup}

\theoremstyle{definition}

\theoremstyle{plain}

\numberwithin{equation}{section}
\numberwithin{proposition}{section}
\numberwithin{conj}{section}	

\usepackage{amsmath}

\usepackage{geometry}
\begin{document}

\title{Mean-Field Control for Diffusion Aggregation system with Coulomb Interaction}
\author{
\sl{Li Chen}\\
   \small\emph{School of Business Informatics and Mathematics, University of Mannheim,} \\
   \small\emph{Mannheim, 68131, Germany}\\
\sl{Yucheng Wang}\\
  \small\emph{School of Business Informatics and Mathematics, University of Mannheim,} \\
   \small\emph{Mannheim, 68131, Germany}\\
\sl{Zhao Wang}\\
   \small\emph{School of Mathematics and Statistics, Changchun University of Technology, }\\
   \small\emph{Changchun, 130012, China}
    }
\date{}
\footnotetext[4]{E-mail addresses: chen@math.uni-mannheim.de, yucheng.wang@uni-mannheim.de, wangzhao@ccut.edu.cn.}

\maketitle

\begin{abstract}
	The mean-field control problem for a multi-dimensional diffusion-aggregation system with Coulomb interaction (the so called parabolic elliptic Keller-Segel system) is considered. The existence of optimal control is proved through the $\Gamma$-convergence of the corresponding control problem of the interacting particle system. There are three building blocks in the whole argument. Firstly, for the optimal control problem on the particle level, instead of using classical method for stochastic system, we study directly the control problem of high-dimensional parabolic equation, i.e. the Liouville equation of it. Secondly, we obtain a strong propagation of chaos result for the interacting particle system by combining the convergence in probability and relative entropy method. Due to this strong mean field limit result, we avoid giving compact support requirement for control functions, which has been often used in the literature. Thirdly, because of strong aggregation effect, additional difficulties arise from control function in obtaining the well-posedness theory of the diffusion-aggregation equation, so that the known method  cannot be directly applied. Instead, we use a combination of local existence result and bootstrap argument to obtain the global solution in the sub-critical regime.
\end{abstract}

\section{Introduction}

Research on the optimal control problem for systems of evolutionary interacting agents and their corresponding mean-field limits attracted a lot of attentions from applied mathematician recently. A number of applications can be found in physics, biology, ecology, economy, and social sciences. Since the number of interacting agents is in general very large, it is hard to follow all the trajectories both numerically and analytically. Therefore, control problem of their mean-field limit equation can be used as an approximation of the control problem for the interacting particle system with mean-field structure. The motivation of studying the control problem for mean-field particle system is the following: by adding the control function and carefully choosing cost functional to the interacting particle system, the group behavior of the agents can be oriented to the designed performance with minimal cost. There are significant progresses in rigorously proving this limit in the last decade, whereby exact references will be given in the next paragraphs.

Since the newly developed mean-field control theory strongly depends on the propagation of chaos result of interacting particle system, we first shortly review the state of art on the mean-field limit result without control. The results for interacting particle systems with smooth interactions have been obtained by applying classical methods based on the direct trajectory estimates, for example in \cite{L31, snitzman_propagation_of_chaos}. For singular interacting forces, which happen in most of the applications, one can find a large number of literature. It is out of the scope of this paper to review all the important results, instead, we focus only on the results which are directly related to the approach we use here. In \cite{L45} Oelschl\"ager successfully applied  moderate interacting particle systems to achieve local non-linear partial differential equations. Stevens investigated the moderate interaction system and its limit to Keller-Segel system \cite{L53}.  Recently, this moderate interacting approach  has been widely used in studying different systems in \cite{L49, L26, L13, CDJ2019, LCZ}. For Coulomb interaction,  Lazarovici and Pickl \cite{lazarovici2017mean} deduced propagation of chaos by proving the convergence in probability of the trajectories.
The relative entropy method has been introduced in \cite{JW2016,JabinWang2018} by Jabin and Wang to handle non-Lipschitz force. Serfaty considered in \cite{Serfaty_2020} the modulated energy approach to derive mean-field limits for repulsive Coulomb interaction potential. Bresch, Jabin, and Wang combined, in \cite  {BreschDidierJabinWangZhenfu2019}, the tools from \cite{JabinWang2018} and the modulated potential energy method from \cite{Serfaty_2020}, and obtained mean-field limits for two dimentional Keller-Segel systems or logarithm potentials in higher dimensions. Recently, Chen, Holzinger, and J\"ungel in \cite{CHJ24} obtained the quantitative estimate for smoothed-$L^2$ norm for moderate interacting system with attractive sub-Coulomb force, which implied not only the propagation of chaos result, but also the fluctuation result in the central limit theorem meaning. In \cite {ChenHolzingerHuo2023}, Chen, Holzinger, and Huo established a connection between the relative entropy and smoothed-$L^2$ estimate for Coulomb and porous medium case. As a result, one can obtain the propagation of chaos result in the strong $L^1$ sense. This method is going to be applied in the mean-field control problem in this paper.

Optimal control problems are naturally considered within the mean-field framework because
of the incompleteness of the self-organization, such as the pattern formation is not being guaranteed.
An effective way is to intervene some individuals towards the pattern formation.
Mean-field optimal control considers the mean-field approximation of the optimal control of
the N-particle system and shows that the optimal control of the N-particle system converging to the
optimal control of the mean-field approximation when the number of the particles goes to infinity. 
The research on obtaining rigorous analytical results of mean-field control problem under different settings has been started only recently.
For deterministic models, Fornasier and Solombrino in \cite{1} introduced the concept of the Mean-field optimal control and got the consistency result by the $\Gamma$-convergence techniques, \cite{2} extended the result to 1st order problems. And \cite{3} obtained some equivalent results among the various representations of mean-field control problems, \cite{5} considered the control problem of mean field PDE on torus and \cite{AM} got the optimal control result with aggregation and distance constraints. \cite{BPTT 2021} provides a framework to compute optimal controls and proves its convergent rate to its mean-field limit control. For the stochastic models, most of researches are focused on McKean-Vlasov optimal control problems, Lacker in \cite{lacker} proved the rigorous consistency with the limit of optimal controls of N-particle systems. \cite{maotan1,maotan2} made extensions to the problems with common noise and \cite{mao} deal with the situation that the state dynamics depend on the joint distribution. \cite{3} got the mean-field optimal control of the nonlocal transport equation with the stochastic term. Further researches on mean-field optimal control can be found in \cite{7,6,YD,BH 2022,BF 2021,8,BCCP,BCC,CDJM,9}.

Many of these interacting systems contain the aggregation effect through attractive interacting pair force among the particles. Adding control to these systems will help in slowing down or speeding up the aggregation effect with minimal costs. 
Most of the known results for mean-field control problems have been focused on non-singular interaction potentials. With singular attractive interactions, i.e. aggregation appears strongly, the interacting particle system has aggregation effect. This implies that  on the macroscopic level the probability density function may blow up in finite time, which is well-known for Keller-Segel system, for example in \cite{DP2004}.

This paper is devoted to the study of the mean-field optimal control of the first order N-particle system with singular interaction force. The mean-field limit of this particle model is the Keller-Segel system, a fundamental system in chemotaxis modeling. The singular potential is given by the fundamental solution of the  ($d\geq 2 $) dimensional Poisson equation. There are actually limited works on optimal control of the Keller-Segel system in the literature on the macroscopic level \cite{10, 11}, which belong to topics on the PDE control problem. These works have been done on bounded domains. In \cite{WWC}, the authors have studied the  mean-field optimal control problem for its two dimensional version, which is a first try for mean-field control problem with singular interaction. Further control problems of chemotaxis systems have been considered in  \cite{a, b}. Except the authors' work \cite{WWC} (written in Chinese for two dimensional case), to our best knowledge, there are no published works on the mean-field optimal control problem with the singular potential yet.

Different from the two dimensional case, where the global existence or blow-up of solution to Keller-Segel system is determined by the initial mass, see \cite{BDP2006, DP2004}, for dimension $d\geq 3$ one can obtain only global existence of solution with small initial data \cite{CPZ04,Perthamebook}. In order to overcome the singularity from interaction effect, the result obtained in this paper focus on the sub-critical  regime of PDE solution.

The choice of control function spaces is motivated by the known works for mean-field control problems in \cite{6,2,1}. The control function space is given by
\begin{equation*}
	{\mathbb{X}}=\Big\{f: \|f\|_{W^{1, q}(\mathbb{R}^d)\cap L^1{(\mathbb{R}^d)}}\leq l(t),\ l(t)\in L^r(0, T), r> 1+\frac{d}{2}, q>d\Big\}. 
\end{equation*}
Worth to mention is that, comparing to the given literature above, we do not need to assume that the control functions have common compact support. Actually we will develop propagation of chaos result in the strong sense, this allows us to use only weak compactness of the control function space.

We denote $\Phi(|x|)=\frac{C_d}{|x|^{d-2}}$ to be the fundamental solution of the Poisson equation.
 
In the following subsections we propose the setting of the microscopic control problem and the macroscopic control problem separately and describe the main result of this paper.
 
\subsection{Setting of the particle control problem}
In the microscopic level, due to the singularity of the interaction potential $\Phi(|x|)$, we start with a particle system with its regularized version $\widetilde{\Phi}_\varepsilon$ as pair interaction potentials, namely
\begin{equation}\label{Phis}
	\widetilde{\Phi}=
	\begin{cases}
		-C_{d}\frac{1}{|x|^{d-2}}, & |x|\geq 2\varepsilon\\
		-C_{d}(2-d)(2\varepsilon)^{1-d}|x|+C_{d}(1-d)(2\varepsilon)^{2-d}, & |x|<2\varepsilon
	\end{cases},  \qquad \widetilde{\Phi}_\varepsilon=j_\varepsilon\ast \widetilde{\Phi},
\end{equation}
where $\varepsilon$ should be chosen depending on $N$ such that $\varepsilon\rightarrow 0$ as $N\rightarrow\infty$.
Other versions of regularized potential for particle systems are also given in \cite{HLP2020,lazarovici2017mean}. In this paper we choose the above regularization to simplify the discussion on the mean-field limit part, and put more efforts to structure the framework of control problem with singular potential. Actually, we do not expect much extra technical difficulties if better cut-off rates, such as in \cite{HLP2020} with extensive estimates, are applied for the control problem discussed in this paper.

For any given control function $f\in{\mathbb{X}}$, we consider the following interacting particle system
\begin{equation}\label{100}
	\begin{cases}
		 \mathrm{d}X_{i}^{N, \varepsilon}=\frac{\chi}{N}\sum_{j=1}^{N}\nabla\widetilde{\Phi}_{\varepsilon}
		(X_{i}^{N, \varepsilon}-X_{j}^{N, \varepsilon})\, \mathrm{d}t-\chi\nabla\widetilde{\Phi}_{\varepsilon}\ast f(X_{i}^{N, \varepsilon})\, \mathrm{d}t+\sqrt{2}\, \mathrm{d}W_{t}^{i},\\
		X_{i}^{N,\varepsilon}(0)=\xi_{i}.
	\end{cases}
\end{equation}
Here we use the standard probability setting for SDEs, i.e.  $(W_t^i)_{i=1}^N$ are independent Brownian motions on $\mathbb{R}^d$ and the initial data $(\xi_i)_{i=1}^N$ are assumed to be i.i.d random variables with probability density $\rho_0$. $\chi$ is a given constant which describes the strength of chemotaxis.  According to the classical SDE theory, there exists a unique square integrable solution $X^{N,\varepsilon}_i[f]$. We denote in this paper the empirical measure $\mu_N[f]=\frac{1}{N}\sum_{i=1}^N\delta_{X^{N,\varepsilon}_i[f]}$, where $\delta_{X^{N,\varepsilon}_i[f]}$ is the Dirac Delta distribution concentrated at ${X^{N,\varepsilon}_i[f]}$.

For any given distribution with density function $z\in L^1(\mathbb{R}^d)\cap L^p(\mathbb{R}^d)$ with $p\in [2,\infty)$, we consider the following cost functional
\begin{equation}\label{27}
	J_{N}(\mu_N[f],f)=\int^T_0\mathbb{E}\big(\big\|j_{\varepsilon}\ast\mu_N[f]-z\big\|_{L^p(\mathbb{R}^d)}\big)\,\mathrm{d}t+\int_{0}^{T}\mathbb{E} \big(\langle f,\mu_N[f]\rangle\big)\,\mathrm{d}t.
\end{equation}
The minimization of this cost functional describes that in the expectation sense, the smoothed empirical measure should be close to the given distribution $z$ with minimal cost from the control function $f$.

Then the goal in the particle part is to obtain, for any given $N$, the existence of $f_N\in {\mathbb{X}}$  such that
\begin{equation}\label{minNproblem}
              J_{N}(\mu_N[f_N],f_N)=\inf_{f\in {\mathbb{X}}} J_{N}(\mu_N[f],f).
\end{equation}
The precise statement of this result and its proof will be given in Section \ref{Nproblem}.

\subsection{Setting of the PDE control problem and the main result}
For any given control function $f\in {\mathbb{X}}$, the expected mean-field limit of \eqref{100} is
\begin{equation}\label{1}
\begin{cases}
\partial_{t}\rho=\Delta\rho-\nabla\cdot(\rho\chi\nabla c),\\
-\Delta c=\rho-f,
\end{cases}
\end{equation}
which is equivalent to
\begin{equation}\label{16}
\partial_{t}\rho=\Delta\rho-\nabla\cdot(\rho\chi\nabla\Phi\ast(\rho-f)).
\end{equation}
Under appropriate condition of the initial data $\rho|_{t=0}=\rho_{0}$, we will show in section \ref{PDEpart} that \eqref{16} has a solution $\rho[f]$. 

Parallel to the particle case, for any given density of a distribution $z\in  L^1(\mathbb{R}^d)\cap L^p(\mathbb{R}^d)$ with $p\in [2, \infty)$, the macroscopic cost functional is defined by
\begin{equation}\label{28}
	J(\rho[f], f)=\int_{0}^{T}\|\rho[f]-z\|_{L^p(\mathbb{R}^d)} \mathrm{d}t+\int_{0}^{T}\langle f,\rho[f]\rangle \mathrm{d}t.
\end{equation}
On the PDE level, the minimization of this cost functional describes that the solution $\rho[f]$ should be close to the given distribution $z$ with minimal cost from the control function $f$. Optimal control of this problem is to find an $\overline f\in {\mathbb{X}}$ such that
\begin{equation}\label{minPDEproblem}
	J(\rho[\overline{f}], \overline{f})=\inf_{f\in {\mathbb{X}}} J(\rho[f],f).
\end{equation}
We obtain the existence of optimal control through mean-field limit of the particle control problem \eqref{minNproblem}. The main result of this paper is

\begin{theorem}\label{mainthm}
	Assume the initial data $\rho_0\geq 0$ satisfies \begin{equation}\label{29}
		\rho_{0}\in L^{1}(\mathbb{R}^d; (1+|x|^2)\,  \mathrm{d}x)\cap W^{1, q}(\mathbb{R}^d),\quad q>d,
			\end{equation}
    then for any fixed $\varepsilon$ and $N$, the cost function $J_{N}(\mu_N[f], f)$ has a minimizer $\widetilde{f}\in {\mathbb{X}}$, i.e.
	\begin{equation*}
		J_{N}(\mu_N[\widetilde{f}], \widetilde{f})=\min_{f\in {\mathbb{X}}} J_{N}(\mu_N[f],f),
	\end{equation*}
	where $\mu_N[f]$ is the empirical measure of the solution of \eqref{100}.
	
In addition, for all $ T>0$ and $ f\in  {\mathbb{X}}$, there exists a constant $\Theta$, such that if
	\begin{align}\label{rho0}
		\big\|\rho_{0}\big\|_{L^{\frac{d}{2}}(\mathbb{R}^d)}^{\frac{d}{2}}\leq\Theta,
			\end{align}
	then there is a solution of \eqref{16}, $\rho[f]\in L^\infty(0,T; L^{\frac{d}{2}}(\mathbb{R}^d)\cap W^{1, q}(\mathbb{R}^d))$.
	\vskip3mm
Furthermore, there exists $\beta_*>0$, for $\varepsilon=N^{-\beta}$ with $ \beta\in(0, \beta_*)$, let $f_{N}\in {\mathbb{X}}$ be any minimizer of $J_{N}(\mu_N[f], f)$, then any weak accumulation point  $\overline{f}$ of $(f_N)_{N\in\mathbb{N}}$ is a minimizer of $J(\rho[f],f)$, namely
	\begin{equation*}
		J(\rho[\overline{f}], \overline{f})=\min_{f\in {\mathbb{X}}}J(\rho[f],f).
	\end{equation*}
\end{theorem}

\begin{remark}
The detailed expression of $\Theta$ will be given in Section \ref{PDEpart}.
\end{remark}

\subsection{Discussion and arrangement of the paper}
The main result shows that we obtain the existence of the minimizer of $J(\rho[f],f)$ not from direct analysis of the PDE control problem, but through proceeding the limit of the particle control problems. The main purpose is to describe the macroscopic control problem in a microscopic level which helps in understanding the same control problem in different scales.

\vskip3mm

The key steps in proving Theorem \ref{mainthm} include:
\begin{enumerate}
	\item Solvability of the particle control problem \eqref{minNproblem}. {Instead of directly applying the classical result for SDE control problems in \cite{HL90}, which has been used in \cite{lacker} for the particle control problem, we will study the corresponding control problem for the Liouville equation.} This result is given in Section \ref{Nproblem}.
	
	\item Well-posedness of initial value problem of diffusion aggregation equation \eqref{16}. This result provides additional estimates of the solution, which can in some sense absorb the singularity of the potential in the mean-field limit discussion. Due to the strong aggregation result, finite time blow-up is expected \cite{BDP2006,Perthamebook,CPZ04}. Therefore, we work only with the sub-critical case, namely an additional restriction on the initial data \eqref{rho0} is given, similar to the case for multi-dimensional Keller-Segel system in \cite{CPZ04}. Unfortunately, due to the appearance of additional control function, the argument given in \cite{CPZ04} does not work directly. Instead, after proving a local existence result, we use the bootstrap argument to extend the solution to any given positive time $T>0$.
	
	\item The mean-field limit discussion for given control function. Due to the singularity of interaction force, we introduce an intermediate problem, the mean-field limit of \eqref{100} for fixed $\varepsilon$ such as has been done in several literature \cite{Oeschlager1984,CDJ2019,CHJZ2021,ChenHolzingerHuo2023}. The mean-field limit result is given in two formulations, i.e. convergence in probability in the sense given in \cite{lazarovici2017mean} and the $L^1$ strong convergence by estimating the relative entropy \cite{JW2016,ChenHolzingerHuo2023}. The novelty at this step is the integration of convergence in probability into the relative entropy estimate, so that one avoids the large deviation estimate provided in \cite{JabinWang2018}. Both results on the mean-field limit discussion will be used to prove the continuity of cost functionals. This part is given in Section \ref{meanfield}.
	
	\item Limits of cost functional and proof of Theorem \ref{mainthm}. There are several different statements in the discussion on the limits related to cost functional, which will be applied in the so-called $\Gamma-$convergence strategy. Due to the strong $L^1$ convergence in the mean-field limit result we can obtain, some intermediate steps in $\Gamma-$convergence framework can be much simplified. More details are given in Section \ref{proofmainthm} in the proof of the main theorem.
\end{enumerate}

As a summary, this paper aims to provide a mean-field control theory for system with singular aggregation effect. The strong result in obtaining the propagation of chaos makes it possible to get rid of the usual requirement that control functions share the same compact support, for example in \cite{6,2,1}.

\section{Optimal control problem for N-particle system}\label{Nproblem}

In this section we will prove that the particle control problem \eqref{minNproblem} has a solution $f_N\in {\mathbb{X}}$. As has been mentioned in the introduction, we work directly on the corresponding high dimensional PDE control problem, not with the SDE system \eqref{100}. After applying It\^{o}'s formula and taking expectation of it, one can easily obtain that the joint distribution of $X_{i}^{N, \varepsilon}$, $i=1,\cdots,N$, denoted by $\rho^{N,\varepsilon}$, satisfies the following linear parabolic PDE in the sense of distribution:
\begin{equation}\label{22}
		\partial_{t}\rho^{N, \varepsilon}-\sum_{i=1}^{N}\Delta_{x_{i}}\rho^{N, \varepsilon}+\sum_{i=1}^{N}\nabla_{x_{i}}\cdot\Big(\chi\rho^{N, \varepsilon}\big(\frac{1}{N}\sum_{j=1}^{N}\nabla\widetilde{\Phi}_{\varepsilon}(x_{i}-x_{j})-\nabla\widetilde{\Phi}_{\varepsilon}\ast f(x_{i})\big)\Big)=0,		
\end{equation}
given initial data $\rho^{N, \varepsilon}|_{t=0}=\rho_{0}^{\otimes N}=\rho_{0}(x_{1})\rho_{0}(x_{2})\cdots\rho_{0}(x_{N})$.
Since for fixed $N,\varepsilon$, the coefficients appeared in \eqref{22} are smooth, one can apply classical theory on Cauchy problem of second order parabolic PDE and obtain that \eqref{22} has a unique classical solution $\rho^{N,\varepsilon}[f]$. 

For any given distribution with density function $z\in  L^1(\mathbb{R}^d)\cap L^p(\mathbb{R}^d)$ with $p\in [2,\infty)$, because of the symmetric property of $\rho^{N,\varepsilon}[f](t,x_1,\cdots,x_N)$, i.e. $\forall i\neq j$,
$$
\rho^{N,\varepsilon}[f](t,x_1,\cdots,x_i,\cdots,x_j,\cdots, x_N)=\rho^{N,\varepsilon}[f](t,x_1,\cdots,x_j,\cdots,x_i,\cdots, x_N),
$$ 
the cost functional given in \eqref{27} can be rewritten into 
\begin{align}	\nonumber &J_{N}(\mu_N[f],f)=\widetilde J_N(\rho^{N,\varepsilon}[f],f)
	\\
\nonumber=&\int^T_0\int_{\mathbb{R}^{dN}}\Big\|\frac{1}{N}\sum_{i=1}^Nj_{\varepsilon}(\cdot-x_i)-z(\cdot)\Big\|_{L^p(\mathbb{R}^d)}\rho^{N,\varepsilon}[f](t,x_1,\cdots,x_N) \mathrm{d}x_1\cdots  \mathrm{d}x_N\mathrm{d}t\\
		\label{NparticleJ}&\hspace{5mm}+\int_{0}^{T}\int_{\mathbb{R}^{d}} f(t,x)\rho^{N,\varepsilon;1}[f](t,x) \mathrm{d}x \mathrm{d}t,
\end{align}
where $\rho^{N,\varepsilon;1}[f]$ is the $1$-st marginal density of $\rho^{N,\varepsilon}[f]$, i.e.
$$
\rho^{N,\varepsilon;1}[f](t,x)=\int_{\mathbb{R}^{d(N-1)}}\rho^{N,\varepsilon}[f](t,x,x_2,\cdots,x_N)  \mathrm{d}x_2\cdots  \mathrm{d}x_N.
$$ It satisfies the equation, by integrating equation \eqref{22} in $x_2,\cdots, x_N$,
\begin{equation}\label{23}
	\partial_{t}\rho^{N, \varepsilon; 1}[f]-\Delta\rho^{N,\varepsilon; 1}[f]+\nabla\cdot\Big(\chi\int_{\mathbb{R}^d}\rho^{N,\varepsilon; 2}[f](x, y)\big(\nabla\widetilde{\Phi}_{\varepsilon}(x-y)-\nabla\widetilde{\Phi}_{\varepsilon}\ast f(x)\big)\,  \mathrm{d}y\Big)=0.
\end{equation}

The main result of this section is
\begin{proposition}\label{propNcontrol} Assume that $\rho_{0}\in L^{1}(\mathbb{R}^d; (1+|x|^2)\,  \mathrm{d}x)\cap L^{\infty}(\mathbb{R}^d)$, then there exists $\widetilde f\in {\mathbb{X}}$ such that 
	$$
	\widetilde J_N(\rho^{N,\varepsilon}[\widetilde f],\widetilde f)=\inf_{f\in {\mathbb{X}}} \widetilde J_N(\rho^{N,\varepsilon}[f],f).
	$$
\end{proposition}

\begin{proof}
	Let $f_{k}\in {\mathbb{X}}$ be a minimization sequence of this optimization problem, namely
	$$
	\lim_{k\rightarrow\infty}\widetilde J_N(\rho^{N,\varepsilon}[f_k],f_k)=\inf_{f\in {\mathbb{X}}} \widetilde J_N(\rho^{N,\varepsilon}[f],f).
	$$
	Since ${\mathbb{X}}$ is bounded subset of $L^r(0,T;W^{1,q}(\mathbb{R}^d))$, there is a weakly convergent subsequence, not relabeled, with $\widetilde{f}$ as its weak limit in ${\mathbb{X}}$. It remains to prove that, we use $f_k$ to be the not relabeled subsequence, 
	\begin{equation*}
		\lim_{k\to\infty}\widetilde J_{N}(\rho^{N,\varepsilon}[f_k], f_k)=\widetilde J_{N}(\rho^{N,\varepsilon}[\widetilde f],\widetilde f).
	\end{equation*}
We divide the rest of the proof into following steps:

{\it Step 1.} Prove that 
	\begin{equation}\label{rhoNstrongN}
		\rho^{N, \varepsilon}[f_{k}]\longrightarrow\rho^{N, \varepsilon}[\widetilde{f}]\ \text{in}\ L^2(0,T;L^1(\mathbb{R}^{dN})),\quad k\rightarrow \infty.
	\end{equation}
Actually, in order to take limit in the first term of \eqref{NparticleJ}, we only need weak convergence of $\rho^{N, \varepsilon}[f_{k}]$. We show here the strong convergence since it comes automatically from the uniform estimate and compactness argument, which does not need extra efforts. The standard $L^2$ estimate of the second order linear parabolic equation \eqref{22} implies that 
	\begin{align}\label{L2Nsystem}
		&\|\rho^{N, \varepsilon}[f_{k}]\|_{L^\infty(0,T;L^2(\mathbb{R}^{dN}))}+\|\rho^{N, \varepsilon}[f_{k}]\|_{L^2(0,T;H^1(\mathbb{R}^{dN}))}+\|\partial_{t}\rho^{N, \varepsilon}[f_{k}]\|_{L^2(0,T;H^{-1}(\mathbb{R}^{dN}))}\leq C(N,\varepsilon, T).
	\end{align}
For the second moment estimate, we multiply the equation \eqref{22} by $\sum_{i=1}^N|x_i|^2$ and obtain
\begin{align*}
		&\frac{ \mathrm{d}}{\mathrm{d}t}\int_{\mathbb{R}^{dN}} \sum_{i=1}^N|x_i|^2\rho^{N, \varepsilon}[f_k]\,  \mathrm{d}x_1\cdots  \mathrm{d}x_N-\sum_{i,m=1}^{N}\int_{\mathbb{R}^{dN}}|x_i|^2\Delta_{x_{m}}\rho^{N, \varepsilon}[f_k]\,  \mathrm{d}x_1\cdots  \mathrm{d}x_N\\
		&\quad +\sum_{i,m=1}^{N}\int_{\mathbb{R}^{dN}}|x_i|^2\nabla_{x_{m}}\cdot\Big(\chi\rho^{N, \varepsilon}[f_k](\frac{1}{N}\sum_{j=1}^{N}\nabla\widetilde{\Phi}_{\varepsilon}(x_{i}-x_{j})-\nabla\widetilde{\Phi}_{\varepsilon}\ast f_k)\Big) \,  \mathrm{d}x_1\cdots  \mathrm{d}x_N =0.
\end{align*}
Further estimates by using
$$
\|\nabla\widetilde{\Phi}_{\varepsilon}\|_{L^{\infty}}+\|\nabla\widetilde{\Phi}_{\varepsilon}\ast f_k\|_{L^{\infty}}\leq C\varepsilon^{-(d-1)}+C(d)(\|f_k\|_{L^1}+\|f_k\|_{L^\infty})\leq C\varepsilon^{-(d-1)}+C(d)l(t)
$$
show that
\begin{align*}
		&\frac{ \mathrm{d}}{\mathrm{d}t}\int_{\mathbb{R}^{dN}} \sum_{i=1}^N|x_i|^2\rho^{N, \varepsilon}[f_k]\,  \mathrm{d}x_1\cdots  \mathrm{d}x_N\\
		\leq &2Nd+2\sum_{i=1}^{N}\int_{\mathbb{R}^{dN}}\nabla_{x_{m}}|x_i|^2\cdot\Big(\chi\rho^{N, \varepsilon}[f_k](\frac{1}{N}\sum_{j=1}^{N}\nabla\widetilde{\Phi}_{\varepsilon}(x_{i}-x_{j})-\nabla\widetilde{\Phi}_{\varepsilon}\ast f_k)\Big)\,  \mathrm{d}x_1\cdots  \mathrm{d}x_N\\
		\leq&2Nd+2\chi\big(\|\nabla\widetilde{\Phi}_{\varepsilon}\|_{L^{\infty}}+\|\nabla\widetilde{\Phi}_{\varepsilon}\ast f_k\|_{L^{\infty}}\big)\sum_{i=1}^{N}\int_{\mathbb{R}^{dN}} (1+ |x_{i}|^2)\rho^{N, \varepsilon}[f_k]\,  \mathrm{d}x_1\cdots\mathrm{d}x_N\\
		\leq&C(N, \varepsilon, \chi, d)(1+l(t))\int_{\mathbb{R}^{dN}} \sum_{i=1}^N|x_i|^2\rho^{N, \varepsilon}[f_k]\,  \mathrm{d}x_1\cdots  \mathrm{d}x_N+C(N, \varepsilon, \chi, d)(1+l(t)).
\end{align*}
Using the Gronwall's inequality, we obtain the estimate for second moment. 
The $L^2$ estimate in \eqref{L2Nsystem}, together with the conservation of mass and the boundedness of second moment, i.e.
\begin{align}\label{Y1}
	\big\|\rho^{N, \varepsilon}[f_{k}]\big\|_{L^\infty(0,T;L^1(\mathbb{R}^{dN}))}=1,\quad&\int_{\mathbb{R}^{dN}}\sum_{i=1}^N|x_i|^2\rho^{N, \varepsilon}[f_{k}] \mathrm{d}x_1\cdots  \mathrm{d}x_N\leq C(N,\varepsilon, \chi, d, T),
\end{align}
implies the compact embedding, by applying Lemma \ref{compactembedding}. This allows us to use the Aubin-Lions lemma to obtain a strongly convergent subsequence, not relabeled, such that
	\begin{equation}\label{Y2}
		\rho^{N, \varepsilon}[f_k]\to\rho^{N, \varepsilon},\ \text{in}\ L^2(0,T;L^2 (\mathbb{R}^{dN})).
	\end{equation}
	Actually, this $L^2$ convergence implies also $L^1$ convergence due to the boundedness of second moment in (\ref{Y1}). More precisely,
for arbitrary small $\varepsilon_{1}>0$, we choose $R$ such that the following estimate holds,
	\begin{equation*}
	\begin{split}
	&\|\rho^{N, \varepsilon}[f_{k}]-\rho^{N, \varepsilon}\|_{L^2(0, T; L^1(\mathbb{R}^{dN}))}\\
	\leq&\|\rho^{N, \varepsilon}[f_{k}]-\rho^{N, \varepsilon}\|_{L^2(0, T; L^2(B_{R}))}|B_{R}|+\int_{0}^T\Big(\int_{B_{R}^c}\frac{|x|^2}{R^2}|\rho^{N, \varepsilon}[f_{k}]-\rho^{N, \varepsilon}|\, \mathrm{d}x_{1}\cdots \mathrm{d}x_{N})^2\ \mathrm{d}t\\
	\leq&\|\rho^{N, \varepsilon}[f_{k}]-\rho^{N, \varepsilon}\|_{L^2(0, T; L^2(B_{R}))}|B_{R}|+C(N, \varepsilon,\chi, d, T)/R^2\\
	\leq&\|\rho^{N, \varepsilon}[f_{k}]-\rho^{N, \varepsilon}\|_{L^2(0, T; L^2(B_{R}))}|B_{R}|+\varepsilon_{1}.
	\end{split}
	\end{equation*}
Therefore for this fixed $R$ the first term vanishes after taking $N\to \infty$ limit. And we obtain 
$$
\lim_{N\rightarrow\infty} \|\rho^{N, \varepsilon}[f_{k}]-\rho^{N, \varepsilon}\|_{L^2(0, T; L^1(\mathbb{R}^{dN}))}\leq \varepsilon_{1}.
$$
Since $\varepsilon_{1}$ is arbitrarily given, we have
	\begin{equation*}
	\rho^{N, \varepsilon}[f_k]\to\rho^{N, \varepsilon},\ \text{in}\ L^2(0,T;L^1 (\mathbb{R}^{dN})).
	\end{equation*}
Furthermore, it is easy to obtain that $\rho^{N, \varepsilon}=\rho^{N, \varepsilon}[\widetilde{f}]$ by taking limits term by term in the weak formulation of \eqref{22}. We omit the details of this standard argument. Therefore the strong convergence in \eqref{rhoNstrongN} is obtained.

{\it Step 2.} Prove that 
\begin{equation}\label{rhoNstrong}
	\rho^{N, \varepsilon;1}[f_{k}]\longrightarrow\rho^{N, \varepsilon;1}[\widetilde{f}]\ \text{in}\ L^2(0,T;L^1(\mathbb{R}^{d})),\quad k\rightarrow \infty.
\end{equation}
This strong convergence is necessary for the convergence discussion of the second term in \eqref{NparticleJ}, since we only have weak convergence for $f_k$.

The equation \eqref{23} for $\rho^{N, \varepsilon;1}$ depends also on the second marginal of $\rho^{N,\varepsilon}$, it is not obvious whether one can get similar estimates to the ones for $\rho^{N, \varepsilon}$ in order to proceed the compactness argument. Actually, one can do it similarly. For completeness, we give details here and present the following estimates,
	$$
	\|\rho^{N, \varepsilon; 1}[f_k]\|_{L^{\infty}(0, T; L^1(\mathbb{R}^{d}, (1+|x|^2) \mathrm{d}x))\cap L^{\infty}(0, T; L^2(\mathbb{R}^{d}))\cap L^2(0, T; H^1(\mathbb{R}^{d}))}\leq C(\varepsilon, \chi, d, T).
	$$

The $L^2$ estimate can be obtained through multiplying the equation (\ref{23}) by $\rho^{N, \varepsilon;1}[f_k]$.
	\begin{equation*}
		\begin{split}
			&\frac{1}{2}\frac{ \mathrm{d}}{\mathrm{d}t}\int_{\mathbb{R}^d} (\rho^{N, \varepsilon; 1}[f_k])^2\,  \mathrm{d}x+\int_{\mathbb{R}^d} |\nabla\rho^{N, \varepsilon; 1}[f_k]|^2\,  \mathrm{d}x\\
			=&\chi\int_{\mathbb{R}^d}\int_{\mathbb{R}^d}\rho^{N, \varepsilon;2}[f_k](x,y)\big(\nabla\widetilde{\Phi}_{\varepsilon}(x-y)-\nabla\widetilde{\Phi}_{\varepsilon}\ast f_k(x)\big)\nabla\rho^{N, \varepsilon; 1}[f_k](x)\, \mathrm{d}y  \mathrm{d}x\\
			\leq&\chi\big(\|\nabla\widetilde{\Phi}_{\varepsilon}\|_{L^{\infty}}+\|\nabla\widetilde{\Phi}_{\varepsilon}\ast f_k\|_{L^{\infty}}\big)\|\rho^{N, \varepsilon; 1}[f_k]\|_{L^2}\|\nabla\rho^{N, \varepsilon; 1}[f_k]\|_{L^2}\\
		    \leq&\frac{1}{2}\|\nabla\rho^{N, \varepsilon; 1}[f_k]\|_{L^2}^2+C(\varepsilon, \chi, d)\|\rho^{N, \varepsilon; 1}[f_k]\|_{L^2}^2	\big(\varepsilon^{-2(d-1)}+l^2(t)\big).
		\end{split}
	\end{equation*}
Then by using Gronwall's inequality, we have
	\begin{equation}\label{b200}
		\|\rho^{N, \varepsilon; 1}[f_k]\|_{L^\infty(0,T;L^2(\mathbb{R}^d))}^2+\int_{0}^{T}\|\nabla\rho^{N, \varepsilon; 1}[f_k]\|_{L^2(\mathbb{R}^d)}^2\, \mathrm{d}t\leq C(\varepsilon, \chi, d, T).
	\end{equation}

 For second moment estimate, we multiply the equation (\ref{23}) by $|x|^2$ and integrating over $\mathbb{R}^d$ and obtain
	\begin{equation*}
		\begin{split}
			\frac{ \mathrm{d}}{\mathrm{d}t}\int_{\mathbb{R}^d} |x|^2\rho^{N,\varepsilon; 1}[f_k]\,  \mathrm{d}x=&2d+\int_{\mathbb{R}^d}\Big(\int_{\mathbb{R}^d}\rho^{N,\varepsilon; 2}[f_k](x, y)(\nabla\widetilde{\Phi}_{\varepsilon}(x-y)-\nabla\widetilde{\Phi}_{\varepsilon}\ast f_k(x))\,  \mathrm{d}y\Big)\cdot x\,  \mathrm{d}x\\
			\leq&2d+(\|\nabla\widetilde{\Phi}_{\varepsilon}\|_{L^{\infty}}+\|\nabla\widetilde{\Phi}_{\varepsilon}\ast f_k\|_{L^{\infty}})\int_{\mathbb{R}^d}\rho^{N, \varepsilon; 1}(1+|x|^2)\,  \mathrm{d}x\\
			\leq &C(\varepsilon, d)(1+l(t))\int_{\mathbb{R}^d} |x|^2\rho^{N,\varepsilon; 1}[f_k]\,  \mathrm{d}x+C(\varepsilon, d)(1+l(t)).
		\end{split}
	\end{equation*}
	Using the Gronwall's inequality to the above inequality, we have the estimate for second moment. {According to the structure of the equation (\ref{23}) and the estimate (\ref{b200}), we have $\|\partial_{t}\rho^{N, \varepsilon; 1}[f_{k}]\|_{L^{2}(0, T; H^{-1}(\mathbb{R}^{d}))}\leq C(\varepsilon, d, T)$}. Then similar to the first step, we obtain the strong convergence \eqref{rhoNstrong} by using Aubin-Lions lemma, the second moment is bounded and Lemma \ref{compactembedding}.
{\it Step 3.} Prove the convergence
\begin{equation}
	\label{JNconvergence}
	\widetilde J_N(\rho^{N,\varepsilon}[f_k],f_k)\longrightarrow \widetilde J_N(\rho^{N,\varepsilon}[\widetilde{f}],\widetilde{f}),\qquad k\rightarrow \infty.
\end{equation}
	According to the definition of the cost function \eqref{NparticleJ}, notice that for given $N,\varepsilon$,
	$$
	\Big\|\frac{1}{N}\sum_{i=1}^Nj_{\varepsilon}(\cdot-x_i)-z(\cdot)\Big\|_{L^p(\mathbb{R}^d)}\in L^\infty(\mathbb{R}^{dN})
	$$
	by the results obtained in {\it step 1} and {\it 2}, we have
	\begin{equation*}
		\begin{split}
			&\widetilde J_N(\rho^{N,\varepsilon}[f_k],f_k)-\widetilde J_N(\rho^{N,\varepsilon}[\widetilde{f}],\widetilde{f})\\
			=&\int^T_0\int_{\mathbb{R}^{dN}}\Big\|\frac{1}{N}\sum_{i=1}^Nj_{\varepsilon}(\cdot-x_i)-z(\cdot)\Big\|_{L^p(\mathbb{R}^d)}(\rho^{N, \varepsilon}[f_{k}]-\rho^{N, \varepsilon}[\widetilde{f}])(t,x_1,\cdots,x_N) \mathrm{d}x_1\cdots  \mathrm{d}x_N\mathrm{d}t\\
			&\hspace{5mm}+\int_{0}^{T}\int_{\mathbb{R}^{d}} \Big(f_k\rho^{N,\varepsilon;1}[f_k]-\widetilde{f}\rho^{N,\varepsilon;1}[\widetilde{f}]\Big) \mathrm{d}x \mathrm{d}t\\
			&\longrightarrow 0\quad\text{as}\quad k\longrightarrow\infty,
              \end{split}
	\end{equation*}	
	where in the last step we have used the weak convergence of $f_k$ and the strong convergence of $\rho^{N,\varepsilon}[f_k]$ and $\rho^{N,\varepsilon;1}[f_k]$.
	
\end{proof}

\section[Existence of solutions]{Existence of solutions for diffusion-aggregation equation}\label{PDEpart}

In this section, we focus on proving the existence result of the limiting diffusion-aggregation equation and the regularized intermediate PDE for given control function $f$. Since we don't add positivity assumption of $f$, the analysis from known well-posedness result of Keller-Segel system in the sub-critical case, for example in \cite{Perthamebook,CPZ04}, does not work. Instead, we provide a bootstrap argument to achieve the global existence result. On top of that, we obtain more regularity of the solution to get estimates for the difference between intermediate solution and the PDE solution.
	
This section is divided into three subsections. In the first subsection, we give the global existence of the weak solution to the system (\ref{1}) with minimum assumption on the initial data. We first prove the local existence in a sub-critical $L^p$ ($p>\frac{d}{2}$) space by using fixed point theory. With additional assumption that the critical $L^{\frac{d}{2}}$ norm of initial data is small, we can proceed the bootstrap argument to extend the solution globally.  In the second subsection, based on the global weak solution obtained from subsection one, if the initial data is in $W^{1,q}(\mathbb{R}^d)$, we establish the same regularity estimate for the solution to the limiting diffusion-aggregation equation. This regularity is used finally in the third subsection to obtain an error estimate of the difference between the limiting diffusion-aggregation equation and the regularized intermediate PDE, i.e. an estimate for $\rho^{\varepsilon}-\rho$.

\vskip3mm
For convenience,in this section we omit the dependence of $[f]$, i.e. $\rho[f]=\rho$ and $\rho^\varepsilon[f]=\rho^\varepsilon$.

\subsection[The global existence of the solution]{The global existence of weak solution} 

We denote
\begin{equation*}
C_{0}:=\|\rho_{0}\|_{L^{\frac{d}{2}}(\mathbb{R}^d)}^{\frac{d}{2}},\ Q_{0}:=\|\rho_{0}\|_{L^{\frac{d}{2}+\eta}(\mathbb{R}^d)}^{\frac{d}{2}+\eta},\ Q=\int_{0}^{\infty}\|f\|_{L^1\cap W^{1, q}}^r\, \mathrm{d}t.
\end{equation*}

We start with the local existence result.

\begin{proposition}\label{PDEest}
For any $\eta\in (|\frac{d}{2}-2|, \frac{d}{2})$, assume the initial data $\rho_0$ satisfy
\begin{equation}\label{b201}
\rho_{0}\in L^\frac{d}{2}(\mathbb{R}^d)\cap L^{\frac{d}{2}+\eta}(\mathbb{R}^d),
\end{equation}
then there exists a time $T^*(Q_{0})>0$ such that the Cauchy problem \eqref{16} has a unique solution satisfying
\begin{equation*}
	\rho\in L^\infty (0,T^*; L^{\frac{d}{2}}(\mathbb{R}^d)\cap L^{\frac{d}{2}+\eta}(\mathbb{R}^d)),
\end{equation*}
\begin{equation*}
		\sup_{0\leq t\leq T^*}\|\rho\|_{L^{\frac{d}{2}+\eta}}^{\frac{d}{2}+\eta}+\int_{0}^{T^*}\|\nabla\rho^{(\frac{d}{4}+\frac{\eta}{2})}\|_{L^2}^2\, \mathrm{d}t\leq 2Q_{0}
		\end{equation*}
and
                \begin{equation*}
       \sup_{0\leq t\leq T^*}\|\rho\|_{L^{\frac{d}{2}}}^{\frac{d}{2}}+\int_{0}^{T^*}\|\nabla\rho^{\frac{d}{4}}\|_{L^2}^2\, \mathrm{d}t\leq 2C_{0}.
        \end{equation*}
\end{proposition}
As it is well-known that for multi-dimentional Keller-Segel system, $\frac{d}{2}$ is the critical norm to obtain global existence of solutions for example \cite{Perthamebook}, due to the additional control function, we choose a subcritical space. We will prove first the local existence by using Banach fixed point theorem in space $L^\infty(0,T^*;L^{\frac{d}{2}+\eta}(\mathbb{R}^d))$. 
	\begin{lemma}\label{suPDE3d}
		Assume that the initial data $\rho_0$ satisfy the condition (\ref{b201}), then there exists a time $\hat{T}^*(Q_{0})>0$ such that the Cauchy problem \eqref{16} has a unique solution satisfying
		\begin{equation*}
			\rho\in L^\infty (0,\hat{T}^*; L^{\frac{d}{2}+\eta}(\mathbb{R}^d))
		\end{equation*}
		and
		\begin{equation*}
		\sup_{0\leq t\leq \hat{T}^*}\|\rho\|_{L^{\frac{d}{2}+\eta}}^{\frac{d}{2}+\eta}+\int_{0}^{\hat{T}^*}\|\nabla\rho^{(\frac{d}{4}+\frac{\eta}{2})}\|_{L^2}^2\, \mathrm{d}t\leq 2Q_{0},
		\end{equation*}
		where $\eta\in (|\frac{d}{2}-2|, \frac{d}{2})$.
	\end{lemma}
	
	\begin{proof}
	We consider the following space	
		\begin{equation*}
			\mathbb{Y}=\{\rho:\|\rho\|_{B_T}\leq 2Q_{0}\},
		\end{equation*}
		where
		\begin{equation*}
			\|\rho\|_{B_T}:=\sup_{0\leq t\leq T}\|\rho\|_{L^{\frac{d}{2}+\eta}}^{\frac{d}{2}+\eta}+{\int_{0}^{T}\|\nabla\rho^{(\frac{d}{4}+\frac{\eta}{2})}\|_{L^2}^2\, \mathrm{d}t}.
		\end{equation*}
		We construct the iterative scheme
		\begin{equation}\label{su2}
			\begin{cases}
				\partial_{t}\rho^{m+1}=\Delta\rho^{m+1}-\nabla\cdot(\rho^{m+1}\chi\nabla c^{m}),\\
				c^{m}=\Phi\ast(\rho^{m}-f),\\
				\rho^{m+1}|_{t=0}=\rho_{0},
			\end{cases}
		\end{equation}
		where $m\geq 0,\ \rho^{0}=0$. In the following, we prove this lemma in the following steps. 
		
		{\it Step 1.} Prove that $(\rho^{m})_{m\in\mathbb{N}}\in\mathbb{Y}$. \\
		We use mathematical induction to prove it. When $m=0$, the first equation of (\ref{su2}) becomes
		\begin{equation}\label{su3}
			\partial_{t}\rho^1=\Delta\rho^1-\nabla\cdot(\rho^1\chi\nabla c^0).
		\end{equation}
		Multiplying the equation (\ref{su3}) by $(\frac{d}{2}+\eta)(\rho^1)^{\frac{d}{2}+\eta-1}$ and integrating with respect to $x$ to get
		\begin{equation}\label{su4}
			\begin{split}
				&\frac{\mathrm{d}}{\mathrm{d}t}\int_{\mathbb{R}^d}(\rho^1)^{\frac{d}{2}+\eta}\, \mathrm{d}x+\frac{4(\frac{d}{2}+\eta-1)}{\frac{d}{2}+\eta}\int_{\mathbb{R}^d} |\nabla(\rho^1)^{\frac{d}{4}+\frac{\eta}{2}}|^2\, \mathrm{d}x\\
				=&-\chi(\frac{d}{2}+\eta-1)\int_{\mathbb{R}^d}\Delta c^0(\rho^1)^{\frac{d}{2}+\eta}\,\mathrm{d}x\\
				=&-\chi(\frac{d}{2}+\eta-1)\int_{\mathbb{R}^d}f(\rho^1)^{\frac{d}{2}+\eta}\,\mathrm{d}x\\
				\leq&\chi(\frac{d}{2}+\eta-1)\|f\|_{L^\infty}\|\rho^1\|_{L^{\frac{d}{2}+\eta}}^{\frac{d}{2}+\eta}.
			\end{split}
		\end{equation}
		Using the Gronwall's inequality to (\ref{su4}), we have
		\begin{equation*}
			\|\rho^{1}\|_{L^{\frac{d}{2}+\eta}}^{\frac{d}{2}+\eta}\leq Q_{0}\exp\Big\{\int_{0}^{t}\chi(\frac{d}{2}+\eta-1)\|f\|_{L^{\infty}}\, \mathrm{d}s\Big\}
			\leq Q_{0}\exp\Big\{C\chi(\frac{d}{2}+\eta-1)Q^{\frac{1}{p}}T^{(1-\frac{1}{p})}\Big\}\leq\frac{3}{2}Q_0,
		\end{equation*}
		if $T\leq T_{1}:=\Big(\big(C\chi(\frac{d}{2}+\eta-1)Q^{\frac{1}{p}}\big)^{-1}\ln\frac{3}{2}\Big)^{\frac{p}{p-1}}$. Then we have
		\begin{equation}\label{su5}
			\begin{split}
				&\|\rho^1\|_{L^{\frac{d}{2}+\eta}}^{\frac{d}{2}+\eta}+\frac{4(\frac{d}{2}+\eta-1)}{\frac{d}{2}+\eta}\int_{0}^{T}\|\nabla(\rho^1)^{\frac{d}{4}+\frac{\eta}{2}}\|_{L^2}^2\, \mathrm{d}t\\
				\leq&\chi(\frac{d}{2}+\eta-1)\frac{3}{2}Q_0\int_{0}^{T}\|f\|_{L^{\infty}}\, \mathrm{d}t+Q_{0}\\
				\leq&C\chi(\frac{d}{2}+\eta-1)\frac{3}{2}Q_0Q^{\frac{1}{p}}T^{(1-\frac{1}{p})}+Q_{0}\\
				\leq&2Q_{0},
			\end{split}
		\end{equation}
		if $T\leq T_{2}:=\Big(C\chi\big(\frac{d}{2}+\eta-1\big)Q^{\frac{1}{p}}\frac{3}{2}\Big)^{-\frac{p}{p-1}}$. Therefore, we have $\rho^1\in\mathbb{Y}$, if $T\leq\hat{T}_{1}:=\min\{T_{1}, T_{2}\}$. Next, assuming $\rho^{n}\in {\mathbb{Y}}, n=2, 3, \cdots, m$, we will prove $\rho^{m+1}\in\mathbb{Y}$. The first equation of (\ref{su2}) becomes
		\begin{equation}\label{su6}
			\partial_{t}\rho^{m+1}=\Delta\rho^{m+1}-\nabla\cdot(\rho^{m+1}\chi\nabla c^{m}).
		\end{equation}
		Multiplying the equation (\ref{su6}) by $(\frac{d}{2}+\eta)(\rho^{m+1})^{\frac{d}{2}+\eta-1}$ and integrating over $\mathbb{R}^d$ to get
		\begin{equation}\label{su7}
			\begin{split}
				&\frac{\mathrm{d}}{\mathrm{d}t}\int_{\mathbb{R}^d}(\rho^{m+1})^{\frac{d}{2}+\eta}\, \mathrm{d}x+\frac{4(\frac{d}{2}+\eta-1)}{\frac{d}{2}+\eta}\int_{\mathbb{R}^d}|\nabla(\rho^{m+1})^{\frac{d}{4}+\frac{\eta}{2}}|^2\, \mathrm{d}x\\
				=&-\chi(\frac{d}{2}+\eta-1)\int_{\mathbb{R}^d}\Delta c^{m}(\rho^{m+1})^{\frac{d}{2}+\eta}\, \mathrm{d}x\\
				=&\chi(\frac{d}{2}+\eta-1)\int_{\mathbb{R}^d}\rho^m(\rho^{m+1})^{\frac{d}{2}+\eta}\, \mathrm{d}x-\chi(\frac{d}{2}+\eta-1)\int_{\mathbb{R}^d}f(\rho^{m+1})^{\frac{d}{2}+\eta}\, \mathrm{d}x\\
				:=&I_{1}+I_{2}.
			\end{split}
		\end{equation}
		First, we estimate $I_{1}$. We have
		\begin{equation*}
			\begin{split}
				I_{1}\leq&\chi(\frac{d}{2}+\eta-1)\|(\rho^{m+1})^{\frac{d}{4}+\frac{\eta}{2}}\|^2_{L^{\frac{d+2\eta}{\frac{d}{2}+\eta-1}}}\|\rho^m\|_{L^{\frac{d}{2}+\eta}}\\
				\leq&\chi(\frac{d}{2}+\eta-1)\|(\rho^{m+1})^{\frac{d}{4}+\frac{\eta}{2}}\|_{L^{2}}^{2-(\frac{d}{\frac{d}{2}+\eta})}\|\nabla(\rho^{m+1})^{\frac{d}{4}+\frac{\eta}{2}}\|_{L^{2}}^{\frac{d}{\frac{d}{2}+\eta}}\|\rho^m\|_{L^{\frac{d}{2}+\eta}}\\
				\leq&\varepsilon_{1}\|\nabla(\rho^{m+1})^{\frac{d}{4}+\frac{\eta}{2}}\|_{L^2}^2+C_{\varepsilon_{1}}\big(\chi(\frac{d}{2}+\eta-1)\big)^{\frac{d+2\eta}{2\eta}}\|(\rho^{m+1})^{\frac{d}{4}+\frac{\eta}{2}}\|_{L^2}^2\|\rho^m\|_{L^{\frac{d}{2}+\eta}}^{\frac{d+2\eta}{2\eta}}.
			\end{split}
		\end{equation*}
		Then we estimate $I_{2}$. We have
		\begin{equation*}
			I_{2}\leq\chi(\frac{d}{2}+\eta-1)\|\rho^{m+1}\|_{L^{\frac{d}{2}+\eta}}^{\frac{d}{2}+\eta}\|f\|_{L^{\infty}}.
		\end{equation*}
		Taking $\varepsilon_{1}=\frac{(\frac{d}{2}+\eta-1)}{\frac{d}{2}+\eta}$, putting the estimates of $I_{1}$ and $I_{2}$ into (\ref{su7}), we have
		\begin{equation*}
			\begin{split}
				&\frac{\mathrm{d}}{\mathrm{d}t}\|\rho^{m+1}\|_{L^{\frac{d}{2}+\eta}}^{\frac{d}{2}+\eta}+\frac{3(\frac{d}{2}+\eta-1)}{\frac{d}{2}+\eta}\int_{\mathbb{R}^d}|\nabla(\rho^{m+1})^{\frac{d}{4}+\frac{\eta}{2}}|^2\, \mathrm{d}x\\
				\leq&C_{\varepsilon_{1}}\big(\chi(\frac{d}{2}+\eta-1)\big)^{\frac{d+2\eta}{2\eta}}\|(\rho^{m+1})^{\frac{d}{4}+\frac{\eta}{2}}\|_{L^2}^2\|\rho^m\|_{L^{\frac{d}{2}+\eta}}^{\frac{d+2\eta}{2\eta}}+\chi(\frac{d}{2}+\eta-1)\|\rho^{m+1}\|_{L^{\frac{d}{2}+\eta}}^{\frac{d}{2}+\eta}\|f\|_{L^{\infty}}.
			\end{split}
		\end{equation*}
		Therefore Gronwall's inequality implies that: if $T$ is chosen such that 
		$$
		T\leq T_{3}:=\min\Big\{\Big(\ln\frac{3}{2}\big(C_{\varepsilon_{1}}\big(\chi(\frac{d}{2}+\eta-1)\big)^{\frac{d+2\eta}{2\eta}}(2Q_{0})^{\frac{1}{\eta}}+\chi(\frac{d}{2}+\eta-1)Q^{\frac{1}{p}}\big)^{-1}\Big)^{\frac{p}{p-1}},\ 1\Big\},
		$$ 
		where $C_{\varepsilon_{1}}=\big(\frac{d+2\eta-2}{d}\big)^{-\frac{d}{2\eta}}\big(\frac{d+2\eta}{2\eta}\big)^{-1}$, then we have
		\begin{equation*}
		\begin{split}
		\|\rho^{m+1}\|_{L^{\frac{d}{2}+\eta}}^{\frac{d}{2}+\eta}\leq&Q_{0}\exp\Big\{\int_{0}^{T}\Big(C_{\varepsilon_{1}}\big(\chi(\frac{d}{2}+\eta-1)\big)^{\frac{d+2\eta}{2\eta}}\|\rho^{m}\|_{L^{\frac{d}{2}+\eta}}^{\frac{d+2\eta}{2\eta}}+\chi(\frac{d}{2}+\eta-1)\|f\|_{L^{\infty}}\Big)\, \mathrm{d}s\Big\}\\
		\leq&Q_{0}\exp\Big\{C_{\varepsilon_{1}}\big(\chi(\frac{d}{2}+\eta-1)\big)^{\frac{d+2\eta}{2\eta}}(2Q_{0})^{\frac{1}{\eta}}T+\chi(\frac{d}{2}+\eta-1)Q^{\frac{1}{p}}T^{1-\frac{1}{p}}\Big\}\\
		\leq&\frac{3}{2}Q_{0}.
		\end{split}
		\end{equation*}
		This implies also
		\begin{equation*}
			\begin{split}
				&\|\rho^{m+1}\|_{L^{\frac{d}{2}+\eta}}^{\frac{d}{2}+\eta}+\frac{3(\frac{d}{2}+\eta-1)}{\frac{d}{2}+\eta}\int_{0}^{T}\|\nabla(\rho^{m+1})^{\frac{d}{4}+\frac{\eta}{2}}\|_{L^2}^2\, \mathrm{d}t\\
				\leq&C_{\varepsilon_{1}}\big(\chi(\frac{d}{2}+\eta-1)\big)^{\frac{d+2\eta}{2\eta}}\|\rho^{m}\|_{L^{\frac{d}{2}+\eta}}^{\frac{d+2\eta}{2\eta}}\|(\rho^{m+1})^{\frac{d}{4}+\frac{\eta}{2}}\|_{L^2}^2T\\
				&\qquad +\chi(\frac{d}{2}+\eta-1)\|\rho^{m}\|_{L^{\frac{d}{2}+\eta}}^{\frac{d}{2}+\eta}\int_{0}^{T}\|f\|_{L^{\infty}}\, \mathrm{d}t+Q_{0}\\
				\leq&Q_{0}T^{1-\frac{1}{p}}\Big(C_{\varepsilon_{1}}\big(\chi(\frac{d}{2}+\eta-1)\big)^{\frac{d+2\eta}{2\eta}}(2Q_{0})^{\frac{1}{\eta}}\frac{3}{2}+\chi(\frac{d}{2}+\eta-1)\frac{3}{2}Q^{\frac{1}{p}}\Big)+Q_{0}\\
				\leq&2Q_{0},
			\end{split}
		\end{equation*}
		if $T\leq T_{4}:=\min\Big\{\Big(C_{\varepsilon_{1}}\big(\chi(\frac{d}{2}+\eta-1)\big)^{\frac{d+2\eta}{2\eta}}(2Q_{0})^{\frac{1}{\eta}}\frac{3}{2}+\chi(\frac{d}{2}+\eta-1)\frac{3}{2}Q^{\frac{1}{p}}\Big)^{-\frac{p}{p-1}}, 1\Big\}$. As a summary, by taking $T\leq\hat{T}_{2}:=\min\{T_{3}, T_{4}\}$, we have $\rho^{m+1}\in\mathbb{Y}$.\\
		
		{\it Step 2.} Prove $(\rho^{m})_{m\in\mathbb{N}}$ is a Cauchy sequence and the limit the unique weak solution of \eqref{16} in $L^{\infty}(0, \hat{T}^*; L^{\frac{d}{2}+\eta}(\mathbb{R}^d))$. 
		
		Actually, the difference $\rho^{m+1}-\rho^{m}$ satisfies
				\begin{equation}\label{su10}
			\partial_{t}(\rho^{m+1}-\rho^{m})=\Delta(\rho^{m+1}-\rho^{m})-\nabla\cdot\big((\rho^{m+1}-\rho^{m})\chi\nabla c^{m}\big)+\nabla\cdot\big(\rho^{m}\chi\nabla(c^{m-1}-c^{m})\big).
		\end{equation}
		By multiplying the equation (\ref{su10}) by $(\frac{d}{2}+\eta)(\rho^{m+1}-\rho^m)^{\frac{d}{2}+\eta-1}$ we get
		\begin{equation}\label{su11}
			\begin{split}
				&\frac{\mathrm{d}}{\mathrm{d}t}\int_{\mathbb{R}^d}(\rho^{m+1}-\rho^{m})^{\frac{d}{2}+\eta}\, \mathrm{d}x+\frac{4(\frac{d}{2}+\eta-1)}{\frac{d}{2}+\eta}\int_{\mathbb{R}^d} |\nabla(\rho^{m+1}-\rho^{m})^{\frac{d}{4}+\frac{\eta}{2}}|^2\, \mathrm{d}x\\
				=&-(\frac{d}{2}+\eta)\chi\int_{\mathbb{R}^d}\nabla\cdot\big((\rho^{m+1}-\rho^{m})\nabla c^m\big)(\rho^{m+1}-\rho^{m})^{\frac{d}{2}+\eta-1}\, \mathrm{d}x\\
				&+(\frac{d}{2}+\eta)\chi\int_{\mathbb{R}^d}\nabla\cdot\big(\rho^{m}\nabla(c^{m-1}-c^{m})\big)(\rho^{m+1}-\rho^{m})^{\frac{d}{2}+\eta-1}\, \mathrm{d}x\\
				:=&J_{1}+J_{2}.
			\end{split}
		\end{equation}
		By integration by parts, Gagliardo-Nirenberg's inequality, H\"older's inequality and Cauchy's inequality, we have
		\begin{equation}{\label{suJ1}}
			\begin{split}
			J_{1}=&-(\frac{d}{2}+\eta-1)\chi\int_{\mathbb{R}^d}\Delta c^{m}(\rho^{m+1}-\rho^m)^{\frac{d}{2}+\eta}\, \mathrm{d}x\\
			=&(\frac{d}{2}+\eta-1)\chi\int_{\mathbb{R}^d}\rho^{m}(\rho^{m+1}-\rho^{m})^{\frac{d}{2}+\eta}\, \mathrm{d}x-(\frac{d}{2}+\eta-1)\chi\int_{\mathbb{R}^d} f(\rho^{m+1}-\rho^{m})^{\frac{d}{2}+\eta}\, \mathrm{d}x\\
			\leq&(\frac{d}{2}+\eta-1)\chi\|\rho^{m}\|_{L^{\frac{d}{2}+\eta}}\|(\rho^{m+1}-\rho^{m})^{\frac{d}{4}+\frac{\eta}{2}}\|_{L^{\frac{2d+4\eta}{d+2\eta-2}}}^2\\
			&+(\frac{d}{2}+\eta-1)\chi\|f\|_{L^{\infty}}\|\rho^{m+1}-\rho^{m}\|_{L^{\frac{d}{2}+\eta}}^{\frac{d}{2}+\eta}\\
			\leq&C\chi(\frac{d}{2}+\eta-1)\|\rho^{m}\|_{L^{\frac{d}{2}+\eta}}\|(\rho^{m+1}-\rho^m)^{\frac{d}{4}+\frac{\eta}{2}}\|_{L^2}^{\frac{4\eta}{d+2\eta}}\|\nabla(\rho^{m+1}-\rho^{m})^{\frac{d}{4}+\frac{\eta}{2}}\|_{L^2}^{\frac{2d}{d+2\eta}}\\
			&+\chi(\frac{d}{2}+\eta-1)\|f\|_{L^{\infty}}\|\rho^{m+1}-\rho^{m}\|_{L^{\frac{d}{2}+\eta}}^{\frac{d}{2}+\eta}\\
			\leq&\varepsilon_{1}\|\nabla(\rho^{m+1}-\rho^m)^{\frac{d}{4}+\frac{\eta}{2}}\|_{L^2}^2+C_{\varepsilon_{1}}\big(\chi(\frac{d}{2}+\eta-1)\big)^{\frac{d+2\eta}{2\eta}}\|\rho^{m}\|_{L^{\frac{d}{2}+\eta}}^{\frac{d+2\eta}{2\eta}}\|(\rho^{m+1}-\rho^{m})^{\frac{d}{4}+\frac{\eta}{2}}\|_{L^2}^2\\
			&+(\frac{d}{2}+\eta-1)\chi\|f\|_{L^{\infty}}\|\rho^{m+1}-\rho^{m}\|_{L^{\frac{d}{2}+\eta}}^{\frac{d}{2}+\eta}
		         \end{split}
		\end{equation}
		and
		\begin{equation}{\label{suJ2}}
			\begin{split}
			J_{2}=&(\frac{d}{2}+\eta)\chi\int_{\mathbb{R}^d}\nabla\cdot\big(\rho^{m}\nabla(c^{m-1}-c^{m})\big)\big((\rho^{m+1}-\rho^{m})^{\frac{d}{4}+\frac{\eta}{2}}\big)^{\frac{2(d+2\eta-2)}{d+2\eta}}\, \mathrm{d}x\\
			=&-C_{0}\chi\int_{\mathbb{R}^d}\rho^{m}\nabla(c^{m-1}-c^{m})\big((\rho^{m+1}-\rho^m)^{\frac{d}{4}+\frac{\eta}{2}}\big)^{\frac{d+2\eta-4}{d+2\eta}}\nabla(\rho^{m+1}-\rho^m)^{\frac{d}{4}+\frac{\eta}{2}}\, \mathrm{d}x\\
			\leq&C_{0}\|\rho^{m}\|_{L^d}\|\nabla(c^{m-1}-c^{m})\|_{L^{\frac{d(d+2\eta)}{d-2\eta}}}\|\big((\rho^{m+1}-\rho^{m})^{\frac{d}{4}+\frac{\eta}{2}}\big)^{\frac{d+2\eta-4}{d+2\eta}}\|_{L^{\frac{2(d+2\eta)}{d+2\eta-4}}}\\
			&\qquad\times\|\nabla(\rho^{m+1}-\rho^{m})^{\frac{d}{4}+\frac{\eta}{2}}\|_{L^2}\\
			\leq&\varepsilon_{1}\|\nabla(\rho^{m+1}-\rho^{m})^{\frac{d}{4}+\frac{\eta}{2}}\|_{L^2}^2+C_{\varepsilon_{1}}C_{0}^2\|(\rho^{m})^{\frac{d}{4}+\frac{\eta}{2}}\|_{L^{\frac{4d}{d+2\eta}}}^{\frac{8}{d+2\eta}}\|\rho^{m-1}-\rho^{m}\|_{L^{\frac{d}{2}+\eta}}^2\\
			&\qquad\times\|\rho^{m+1}-\rho^{m}\|_{L^{\frac{d}{2}+\eta}}^{\frac{d+2\eta-4}{2}}\\
			\leq&C_{\varepsilon_{1}}C_{0}^2\|(\rho^{m})^{\frac{d}{4}+\frac{\eta}{2}}\|_{L^2}^{\frac{2(4-d+2\eta)}{d+2\eta}}\|\nabla(\rho^m)^{\frac{d}{4}+\frac{\eta}{2}}\|_{L^2}^{\frac{2(d-2\eta)}{d+2\eta}}\|\rho^{m}-\rho^{m-1}\|_{L^{\frac{d}{2}+\eta}}^2\|\rho^{m+1}-\rho^m\|_{L^{\frac{d}{2}+\eta}}^{\frac{d+2\eta-4}{2}}\\
			&+\varepsilon_{1}\|\nabla(\rho^{m+1}-\rho^m)^{\frac{d}{4}+\frac{\eta}{2}}\|_{L^2}^2\\
			\leq&C_{\varepsilon_{2}}(C_{\varepsilon_{1}}C_{0}^2)^{\frac{d+2\eta}{4}}\|(\rho^{m})^{\frac{d}{4}+\frac{\eta}{2}}\|_{L^2}^{\frac{4-d+2\eta}{2}}\|\nabla(\rho^m)^{\frac{d}{4}+\frac{\eta}{2}}\|_{L^2}^{\frac{d-2\eta}{2}}\|\rho^m-\rho^{m-1}\|_{L^{\frac{d}{2}+\eta}}^{\frac{d+2\eta}{2}}\\
			&+\varepsilon_{1}\|\nabla(\rho^{m+1}-\rho^m)^{\frac{d}{4}+\frac{\eta}{2}}\|_{L^2}^2+\varepsilon_{2}\|\rho^{m+1}-\rho^{m}\|_{L^{\frac{d}{2}+\eta}}^{\frac{d}{2}+\eta},
			\end{split}
		\end{equation}
where $C_{0}=\chi(d+2\eta-2)$. Putting (\ref{suJ1}) and (\ref{suJ2}) into (\ref{su11}), choosing $\varepsilon_{1}=\frac{\frac{d}{2}+\eta-1}{2(\frac{d}{2}+\eta)}, \varepsilon_{2}=\frac{1}{2}$, we obtain
\begin{align}\label{su13}
	\begin{split}
		&\frac{\mathrm{d}}{\mathrm{d}t}\int_{\mathbb{R}^d} (\rho^{m+1}-\rho^{m})^{\frac{d}{2}+\eta}\, \mathrm{d}x+\frac{3(\frac{d}{2}+\eta-1)}{\frac{d}{2}+\eta}\int_{\mathbb{R}^d} |\nabla(\rho^{m+1}-\rho^{m})^{\frac{d}{4}+\frac{\eta}{2}}|^2\, \mathrm{d}x\\
		\leq&\mathcal{M}\|\rho^{m+1}-\rho^m\|_{L^{\frac{d}{2}+\eta}}^{\frac{d}{2}+\eta}+\mathcal{N}\|\rho^m-\rho^{m-1}\|_{L^{\frac{d}{2}+\eta}}^{\frac{d}{2}+\eta},
	\end{split}
\end{align} 
where 
\begin{align*}
	&\mathcal{M}=1+C_{\varepsilon_{1}}(\chi(\frac{d}{2}+\eta-1))^{\frac{d+2\eta}{2\eta}}\|\rho^m\|_{L^{\frac{d}{2}+\eta}}^{\frac{d+2\eta}{2\eta}}+(\frac{d}{2}+\eta-1)\chi\|f\|_{L^{\infty}},
	\\
	&\mathcal{N}=(C_{\varepsilon_{1}}(\chi(d+2\eta-2))^2)^{\frac{d+2\eta}{4}}\|(\rho^{m})^{\frac{d}{4}+\frac{\eta}{2}}\|_{L^2}^{\frac{4-d+2\eta}{2}}\|\nabla(\rho^m)^{\frac{d}{4}+\frac{\eta}{2}}\|_{L^2}^{\frac{d-2\eta}{2}}.
\end{align*}
We get the following inequality by Gronwall's inequality
\begin{align}\label{su14}
	\begin{split}
		\|\rho^{m+1}-\rho^m\|_{L^{\frac{d}{2}+\eta}}^{\frac{d}{2}+\eta}
		\leq&  e^{\int_0^T\mathcal{M}\mathrm{d}t}\int_0^T\mathcal{N}\|\rho^m-\rho^{m-1}\|_{L^{\frac{d}{2}+\eta}}^{\frac{d}{2}+\eta}\, \mathrm{d}t\\
		\leq&\Big(e^{\int_0^T\mathcal{M}\mathrm{d}t}\int_0^T\mathcal{N}\mathrm{d}t\Big)\|\rho^m-\rho^{m-1}\|_{L^{\frac{d}{2}+\eta}}^{\frac{d}{2}+\eta}.
	\end{split}
\end{align}
Now integrating (\ref{su13}) in $t$ and combining (\ref{su14}), we have
\begin{align}\label{su18}
	\begin{split}
		& \|\rho^{m+1}-\rho^m\|_{L^{\frac{d}{2}+\eta}}^{\frac{d}{2}+\eta}+\frac{3(\frac{d}{2}+\eta-1)}{\frac{d}{2}+\eta}\int_0^T\|\nabla(\rho^{m+1}-\rho^m)^{\frac{d}{4}+\frac{\eta}{2}}\|_{L^2}^2\, \mathrm{d}t  \\
		\leq&\int_0^T \mathcal{M}\|\rho^{m+1}-\rho^m\|_{L^{\frac{d}{2}+\eta}}^{\frac{d}{2}+\eta}\, \mathrm{d}t+\int_0^T \mathcal{N}\|\rho^m-\rho^{m-1}\|_{L^{\frac{d}{2}+\eta}}^{\frac{d}{2}+\eta}\, \mathrm{d}t\\
		\leq&\mathcal{N}T\Big(1+\mathcal{M}Te^{\int_{0}^{T}\mathcal{M}\, \mathrm{d}t}\Big)\|\rho^m-\rho^{m-1}\|_{L^{\frac{d}{2}+\eta}}^{\frac{d}{2}+\eta}\\
		\leq&\mathcal{N}T\Big(1+\mathcal{M}\exp\Big\{{1+C_{\varepsilon_{1}}\big(\chi(\frac{d}{2}+\eta-1)\big)^{\frac{d+2\eta}{2\eta}}(2Q_{0})^{\frac{1}{\eta}}+(\frac{d}{2}+\eta-1)\chi Q^{\frac{1}{p}}}\Big\}\Big)\\
		&\qquad\times\|\rho^m-\rho^{m-1}\|_{L^{\frac{d}{2}+\eta}}^{\frac{d}{2}+\eta}\\
		:=&\mathcal{S}T\|\rho^m-\rho^{m-1}\|_{L^{\frac{d}{2}+\eta}}^{\frac{d}{2}+\eta}.
		\end{split}
\end{align}
By  choosing $T<\hat{T}_{3}:=\mathcal{S}^{-1}$, we have
\begin{align}\label{su19}
	\|\rho^{m+1}-\rho^m\|_{B}\leq C\|\rho^{m}-\rho^{m-1}\|_{B}
\end{align} 
where $0<C<1$, which implies that $(\rho^m)_{m\in\mathbb{N}}$ is a Cauchy sequence. Let $\hat{T}^*=\min\{\hat{T}_1, \hat{T}_2, \hat{T}_3\}$, since the space $\mathbb{Y}$ is complete, this sequence has a limit $\rho\in L^{\infty}(0, \hat{T}^*; L^{\frac{d}{2}+\eta}(\mathbb{R}^d))$. The by taking the limit $m\to\infty$ in \eqref{su2}, we can obtain that $\rho$ is a weak solution.
The uniqueness of the solution can be easily obtained by proceeding the same estimate as in \eqref{su19}.
\end{proof}

In the next, we prove that, within short time the above obtained local solution keeps the boundedness of $L^{\frac{d}{2}}$ norm, which is the critical case. This fact plays also important role later for the bootstrap argument to extend the solution to arbitrary given time.

\begin{lemma}\label{b1}
Assume that the initial data $\rho_{0}$ satisfy the condition (\ref{b201}), then there exist a time $0<T^*<\hat{T}^*$ such that
\begin{equation*}
\rho\in L^{\infty}(0, T^*; L^{\frac{d}{2}}({\mathbb{R}^d}))
\end{equation*}
and
\begin{equation*}
\sup_{0\leq t\leq T^*}\|\rho\|_{L^{\frac{d}{2}}}^{\frac{d}{2}}+\int_{0}^{T^*}\|\nabla\rho^{\frac{d}{4}}\|_{L^2}^2\, \mathrm{d}t\leq 2C_{0}.
\end{equation*}
\end{lemma}

\begin{proof}
Multiplying the equation (\ref{1}) by $\frac{d}{2}\rho^{\frac{d}{2}-1}$ and integrating over $\mathbb{R}^d$, together with the estimates in Lemma \ref{suPDE3d}, we obtain
	\begin{equation}\label{12}
		\begin{split}
			&\frac{\mathrm{d}}{\mathrm{d}t}\int_{\mathbb{R}^d}\rho^{\frac{d}{2}}\,  \mathrm{d}x+4(1-\frac{2}{d})\int_{\mathbb{R}^d} |\nabla\rho^{\frac{d}{4}}|^2\,  \mathrm{d}x\\
			=&\chi(\frac{d}{2}-1)\int_{\mathbb{R}^d} (\rho-f)\rho^{\frac{d}{2}}\,  \mathrm{d}x\\
			\leq&\chi(\frac{d}{2}-1)\|\rho\|_{L^{\frac{d}{2}+\eta}}\|\rho^{\frac{d}{4}}\|_{L^{\frac{2(d+2\eta)}{d+2\eta-2}}}^2+\chi(\frac{d}{2}-1)\|f\|_{L^{\infty}}\|\rho^{\frac{d}{2}}\|_{L^1}\\
			\leq&\chi(\frac{d}{2}-1)\|\rho\|_{L^{\frac{d}{2}+\eta}}\|\rho^{\frac{d}{4}}\|_{L^2}^{\frac{4\eta}{d+2\eta}}\|\nabla\rho^{\frac{d}{4}}\|_{L^2}^{\frac{2d}{d+2\eta}}+\chi(\frac{d}{2}-1)\|f\|_{L^{\infty}}\|\rho\|_{L^{\frac{d}{2}}}^{\frac{d}{2}}\\
			\leq&\varepsilon_{1}\|\nabla\rho^{\frac{d}{4}}\|_{L^2}^2+C_{\varepsilon_{1}}\big(\chi(\frac{d}{2}-1)\|\rho\|_{L^{\frac{d}{2}+\eta}}\|\rho^{\frac{d}{4}}\|_{L^2}^{\frac{4\eta}{d+2\eta}}\big)^{\frac{d+2\eta}{2\eta}}+\chi(\frac{d}{2}-1)\|f\|_{L^{\infty}}\|\rho\|_{L^{\frac{d}{2}}}^{\frac{d}{2}}\\
			\leq&\varepsilon_{1}\|\nabla\rho^{\frac{d}{4}}\|_{L^2}^2+C_{\varepsilon_{1}}\big(\chi(\frac{d}{2}-1)\big)^{\frac{d+2\eta}{2\eta}}(2Q_{0})^{\frac{1}{\eta}}\|\rho\|_{L^{\frac{d}{2}}}^{\frac{d}{2}}+\chi(\frac{d}{2}-1)\|f\|_{L^{\infty}}\|\rho\|_{L^{\frac{d}{2}}}^{\frac{d}{2}}.
			\end{split}
	\end{equation}
Let $\varepsilon_{1}=(1-\frac{2}{d})$, using the Gronwall's inequality to the above inequality, we have
\begin{equation*}
\|\rho\|_{L^{\frac{d}{2}}}^{\frac{d}{2}}\leq C_{0}\exp\Big\{\int_{0}^{T}\Big(C_{\varepsilon_{1}}\big(\chi(\frac{d}{2}-1)\big)^{\frac{d+2\eta}{2\eta}}(2Q_{0})^{\frac{1}{\eta}}+\chi(\frac{d}{2}-1)\|f\|_{L^{\infty}}\Big)\, \mathrm{d}t\Big\}\leq\frac{3}{2}C_{0},
\end{equation*}
if $T\leq T_{5}:=\Big(\ln \frac{3}{2}\big(C_{\varepsilon_{1}}(\chi(\frac{d}{2}-1))^{\frac{d+2\eta}{2\eta}}(2Q_{0})^{\frac{1}{\eta}}+C\chi(\frac{d}{2}-1)Q^{\frac{1}{p}}\big)^{-1}\Big)^{\frac{p}{p-1}}$. Then the inequality (\ref{12}) gives that
\begin{equation*}
\begin{split}
&\|\rho\|_{L^{\frac{d}{2}}}^{\frac{d}{2}}+\frac{3(d-2)}{d}\int_{0}^{T}\|\nabla\rho^{\frac{d}{4}}\|_{L^2}^2\, \mathrm{d}x\\
\leq&\Big(C_{\varepsilon_{1}}\big(\chi(\frac{d}{2}-1)\big)^{\frac{d+2\eta}{2\eta}}(2Q_{0})^{\frac{1}{\eta}}+\chi(\frac{d}{2}-1)Q^{\frac{1}{p}}\Big)\frac{3}{2}C_{0}T^{\frac{p-1}{p}}+C_{0}\\
\leq&2C_{0},
\end{split}
\end{equation*}
if $T\leq T_{6}:=\Big(\frac{2}{3}\big(C_{\varepsilon_{1}}(\chi(\frac{d}{2}-1))^{\frac{d+2\eta}{2\eta}}(2Q_{0})^{\frac{1}{\eta}}+\chi(\frac{d}{2}-1)Q^{\frac{1}{p}}\big)^{-1}\Big)^{\frac{p}{p-1}}$. Then the lemma is proved by taking  $T^*\leq\min\{T_{5}, T_{6}, \hat{T}^*\}$.
\end{proof}

\begin{proof}[The proof of Proposition \ref{PDEest}]
A combination of the results in Lemma \ref{suPDE3d} and \ref{b1} implies the statement of Proposition \ref{PDEest}.
\end{proof}

In the remaining part of this subsection, we use the bootstrap argument to extend the local solution to global solution.

\begin{proposition}\label{a}
Assume the initial data $\rho_{0}$ satisfy the condition (\ref{b201}) and $C_{0}\leq\Theta$, where $\Theta$ is given in \eqref{theta}, then we have
	$$
	\rho\in  L^{\infty}(0, \infty; L^{\frac{d}{2}}(\mathbb{R}^d)\cap L^{\frac{d}{2}+\eta}(\mathbb{R}^d)),
	$$
where  $\eta>0$ is given in Proposition \ref{PDEest}.
\end{proposition}

The following lemma is the key point to finish the proof of Proposition \ref{a}.

\begin{lemma}\label{CLJ}
Assume the initial data satisfy (\ref{b201}). Then there exist a positive constant $\Theta$ such that if the solution of (\ref{1}) on $[0, T]\times\mathbb{R}^d$ satisfy
\begin{equation}\label{15}
		\sup_{0\leq t\leq T}\|\rho\|_{L^{\frac{d}{2}}(\mathbb{R}^d)}^{\frac{d}{2}}
		\leq 2C_{0}^{1/2},\mbox{ for } C_{0}\leq\Theta,
\end{equation}
then the following estimates hold
\begin{equation*}
\sup_{0\leq t\leq T}\|\rho\|_{L^{\frac{d}{2}+\eta}(\mathbb{R}^d)}^{\frac{d}{2}+\eta}\leq\frac{3}{2}Q_{0},\quad\sup_{0\leq t\leq T}\|\rho\|_{L^{\frac{d}{2}}(\mathbb{R}^d)}^{\frac{d}{2}}\leq C_{0}^{1/2}.
\end{equation*}
\end{lemma}

\begin{proof}
Multiplying the equation (\ref{1}) by $(\frac{d}{2}+\eta)\rho^{\frac{d}{2}+\eta-1}$ and integrating over $\mathbb{R}^d$ to get
	\begin{equation*}
	\begin{split}
	&\frac{\mathrm{d}}{\mathrm{d}t}\int_{\mathbb{R}^d}\rho^{\frac{d}{2}+\eta}\, \mathrm{d}x+\frac{4(\frac{d}{2}+\eta-1)}{\frac{d}{2}+\eta}\int_{\mathbb{R}^d}|\nabla\rho^{\frac{d}{4}+\frac{\eta}{2}}|^2\, \mathrm{d}x\\
	=&(\frac{d}{2}+\eta-1)\chi\int_{\mathbb{R}^d}\rho^{\frac{d}{2}+\eta}\rho\, \mathrm{d}x-(\frac{d}{2}+\eta-1)\chi\int_{\mathbb{R}^d}\rho^{\frac{d}{2}+\eta}f\, \mathrm{d}x\\
	\leq&(\frac{d}{2}+\eta-1)\chi\|\rho\|_{L^{\frac{d}{2}}}\|\nabla\rho^{\frac{d}{4}+\frac{\eta}{2}}\|_{L^2}^2+(\frac{d}{2}+\eta-1)\chi\|\rho^{\frac{d}{4}+\frac{\eta}{2}}\|_{L^{\frac{2dr}{dr-2r+2+2\eta}}}^{2}\|f\|_{L^{\frac{dr}{2r-2-2\eta}}}\\
	\leq&(\frac{d}{2}+\eta-1)\chi\|\rho\|_{L^{\frac{d}{2}}}\|\nabla\rho^{\frac{d}{4}+\frac{\eta}{2}}\|_{L^2}^2+(\frac{d}{2}+\eta-1)\chi\|\rho^{\frac{d}{4}+\frac{\eta}{2}}\|_{L^{\frac{2d}{d+2\eta}}}^{\frac{2}{r}}\|\nabla\rho^{\frac{d}{4}+\frac{\eta}{2}}\|_{L^2}^{2(1-\frac{1}{r})}\|f\|_{L^{\frac{dr}{2r-2-2\eta}}}\\
	\leq&(\frac{d}{2}+\eta-1)\chi\|\rho\|_{L^{\frac{d}{2}}}\|\nabla\rho^{\frac{d}{4}+\frac{\eta}{2}}\|_{L^2}^2+\varepsilon_{1}\|\nabla\rho^{\frac{d}{4}+\frac{\eta}{2}}\|_{L^2}^2\\
	&+C_{\varepsilon_{1}}\big(\chi(\frac{d}{2}+\eta-1)\big)^{r}\|f\|_{L^{\frac{dr}{2r-2-2\eta}}}^{r}\|\rho^{\frac{d}{4}+\frac{\eta}{2}}\|_{L^{\frac{2d}{d+2\eta}}}^2\\
	\leq&(\frac{d}{2}+\eta-1)\chi C_0^{\frac{2}{d}}\|\nabla\rho^{\frac{d}{4}+\frac{\eta}{2}}\|_{L^2}^2+\varepsilon_{1}\|\nabla\rho^{\frac{d}{4}+\frac{\eta}{2}}\|_{L^2}^2\\
	&+C_{\varepsilon_{1}}\big(\chi(\frac{d}{2}+\eta-1)\big)^{r}\|f\|_{L^{\frac{dr}{2r-2-2\eta}}}^{r}\|\rho^{\frac{d}{4}+\frac{\eta}{2}}\|_{L^{\frac{2d}{d+2\eta}}}^2.
	\end{split}
	\end{equation*}
If 
\begin{equation}
\label{theta1}
C_{0}\leq\Theta_{1}:=\Big(\frac{1}{2}\big(\frac{\frac{d}{2}+\eta-1}{d+2\eta}(\chi(\frac{d}{2}+\eta-1))^{-1}\big)^{\frac{d}{2}}\Big)^2,
\end{equation}
and $\varepsilon_{1}=\frac{\frac{d}{2}+\eta-1}{d+2\eta}$, we have
\begin{equation*}
\frac{\mathrm{d}}{\mathrm{d}t}\|\rho\|_{L^{\frac{d}{2}+\eta}}^{\frac{d}{2}+\eta}\leq C_{\varepsilon_{1}}\big(\chi(\frac{d}{2}+\eta-1)\big)^{r}\|\rho\|_{L^{\frac{d}{2}}}^{\frac{d+2\eta}{2}}\|f\|_{L^{\frac{dr}{2r-2-2\eta}}}^r.
\end{equation*}
Using the Gronwall's inequality to the above inequality, we have
\begin{equation}\label{bc1999}
\|\rho\|_{L^{\frac{d}{2}+\eta}}^{\frac{d}{2}+\eta}\leq Q_{0}+C_{\varepsilon_{1}}\big(\chi(\frac{d}{2}+\eta-1)\big)^rQ(2C_{0}^{\frac{1}{2}})^{\frac{d+2\eta}{d}}\leq\frac{3}{2}Q_{0},
\end{equation}
where we need to take
\begin{equation}
\label{theta2}
C_{0}\leq\Theta_{2}:=\Big(\frac{1}{2}\big(\frac{Q_{0}}{2}(C_{\varepsilon_{1}}(\chi(\frac{d}{2}+\eta-1))^rQ)^{-1}\big)^{\frac{d}{d+2\eta}}\Big)^2.
\end{equation}
Multiplying the equation (\ref{1}) by $\frac{d}{2}\rho^{\frac{d}{2}-1}$ and integrating over $\mathbb{R}^d$, we get
	\begin{equation}\label{s40}
		\frac{\mathrm{d}}{\mathrm{d}t}\int_{\mathbb{R}^d} \rho^{\frac{d}{2}}\,  \mathrm{d}x+4(1-\frac{2}{d})\int_{\mathbb{R}^d} |\nabla\rho^{\frac{d}{4}}|^2\,  \mathrm{d}x=(\frac{d}{2}-1)\chi\int_{\mathbb{R}^d} \rho^{\frac{d}{2}}(\rho-f)\,  \mathrm{d}x.
	\end{equation}
	By Gagliardo-Nirenberg's inequality, H\"older's inequality and Cauchy's inequality, we have
	\begin{equation*}
		(\frac{d}{2}-1)\chi\int_{\mathbb{R}^d} \rho^{\frac{d}{2}+1}\,  \mathrm{d}x\leq (\frac{d}{2}-1)\chi\|\rho\|_{L^{\frac{d}{2}}}\|\nabla\rho^{\frac{d}{4}}\|_{L^2}^2\leq (\frac{d}{2}-1)\chi C_0^{\frac{2}{d}}\|\nabla\rho^{\frac{d}{4}}\|_{L^2}^2
	\end{equation*}
	and
	\begin{equation*}
		\begin{split}
			(\frac{d}{2}-1)\chi\int_{\mathbb{R}^d}\rho^{\frac{d}{2}}f\,  \mathrm{d}x\leq&(\frac{d}{2}-1)\chi\|\rho^{\frac{d}{2}}\|_{L^{\frac{dr}{(d-2)(r-1)+d}}}\|f\|_{L^{\frac{dr}{2(r-1)}}}\\
			\leq& (\frac{d}{2}-1)\chi\|\rho^{\frac{d}{4}}\|_{L^2}^{\frac{2}{r}}\|\nabla\rho^{\frac{d}{4}}\|_{L^2}^{\frac{2(r-1)}{r}}\|f\|_{L^{\frac{dr}{2(r-1)}}}\\
			\leq&(1-\frac{2}{d})\|\nabla\rho^{\frac{d}{4}}\|_{L^2}^2+C(d, r)\|\rho^{\frac{d}{4}}\|_{L^2}^2\|f\|_{L^{\frac{dr}{2(r-1)}}}^{r},
		\end{split}
	\end{equation*}
	where \begin{equation}\label{C0}
		C(d,r)=\Big(\big(1-\frac{2}{d}\big)\frac{r}{r-1}\Big)^{-(r-1)}r^{-1}\big((\frac{d}{2}-1)\chi\big)^r.
	\end{equation}
	Therefore, by taking
	\begin{equation}
	\label{theta3}C_{0}\leq\Theta_{3}:=\big(\frac{1}{2}(\frac{4(d-2)}{d(d-2)\chi})^{\frac{d}{2}}\big)^2, 
	\end{equation} we have
	\begin{equation*}
		\frac{\mathrm{d}}{\mathrm{d}t}\|\rho\|_{L^{\frac{d}{2}}}^{\frac{d}{2}}\leq C(d,r)\|\rho\|_{L^{\frac{d}{2}}}^{\frac{d}{2}}\|f\|_{L^{\frac{dr}{2(r-1)}}}^{r}.
	\end{equation*}
Using the Gronwall's inequality to the above inequality, we have 
	\begin{eqnarray*}
		\sup_{0\leq t\leq T}\|\rho\|_{L^{\frac{d}{2}}}^{\frac{d}{2}} &\leq& \|\rho_0\|_{L^{\frac{d}{2}}}^{\frac{d}{2}} \exp\Big\{C(d, r)\int^T_0 \|f\|_{L^{\frac{dr}{2(r-1)}}}^{r} \mathrm{d}t\Big\}\\
		&\leq & C_0 \exp\Big\{C(d,r)\int^T_0 \|f\|_{L^1\cap W^{1,q}}^{r}\mathrm{d}t\Big\}\leq C_0\exp\Big\{C(d,r)\|l\|^r_{L^r(0,T)}\Big\}.
	\end{eqnarray*}
	Now by choosing 
	\begin{equation}
	\label{theta4}C_0\leq\Theta_{4}:=\exp\{-2C(d, r)\|l\|_{L^r(0, T)}^r\},
	\end{equation} 
	we obtain that 
	\begin{equation}\label{bc2000}
	\sup_{0\leq t\leq T}\|\rho\|_{L^{\frac{d}{2}}}^{\frac{d}{2}}\leq C_0^{1/2}.
	\end{equation}
As a summary, by taking 
\begin{equation}
\label{theta}C_0\leq \Theta:=\min\{\Theta_{1}, \Theta_{2}, \Theta_{3},\Theta_{4}\}
\end{equation}
where $\Theta_i$ are given in \eqref{theta1},\eqref{theta2},\eqref{theta3}, and\eqref{theta4}, the statement of the lemma follows directly from
a combination of (\ref{bc1999}) and (\ref{bc2000}).
\end{proof}

\begin{proof}[Proof of Proposition \ref{a}]
Let
\begin{equation}\label{assum}
\widetilde{T}=\sup\big\{T|(\ref{15}) \mbox{ holds }\big\}.
\end{equation}
As $\|\rho_{0}\|_{L^{\frac{d}{2}}}^{\frac{d}{2}}=C_{0}\leq C_{0}^{\frac{1}{2}}$ and $\rho_{0}\in L^{\frac{d}{2}+\eta}$, then using the Proposition \ref{PDEest}, we know that there exists $T^*>0$ such that (\ref{15}) holds for $T=T^*$. That is $\widetilde{T}\geq T^*>0$. Then we prove $\widetilde{T}=\infty$. Otherwise, $\widetilde{T}<\infty$. Then by Lemma \ref{CLJ}, we know that
\begin{equation*}
\sup_{0\leq t\leq \widetilde{T}}\|\rho\|_{L^{\frac{d}{2}}}^{\frac{d}{2}}\leq C_{0}^{\frac{1}{2}},\ \sup_{0\leq t\leq\widetilde{T}}\|\rho\|_{L^{\frac{d}{2}+\eta}}^{\frac{d}{2}+\eta}\leq\frac{3}{2}Q_{0}.
\end{equation*}
Then by Proposition \ref{PDEest} again, there exists $T^{**}>\widetilde{T}$ such that (\ref{15}) holds for $T=T^{**}$ which contradicts with (\ref{assum}). So $\widetilde{T}=\infty$. Therefore the proof of Proposition \ref{a} is finished.
\end{proof}

\subsection[Regularity]{Regularity of the weak solution to the system (\ref{1})}
In this subsection, for any given time $0<T<\infty$, we establish the further regularity of the weak solution to the system (\ref{1}) on $[0, T]\times\mathbb{R}^d$. 

\begin{proposition}\label{bca}
If the initial data satisfy (\ref{29}), then for any $T>0$, we have
\begin{equation}\label{b202}
\rho\in L^\infty(0, T; W^{1, q}(\mathbb{R}^d)).
\end{equation}
\end{proposition}

\begin{proof}
In order to get the $L^q$-norm of $\nabla\rho$, we need the $L^{2q}$-norm of $\rho$. We start with the estimate for $L^q$-norm of $\rho$. Multiplying the equation (\ref{1}) by $q\rho^{q-1}$ and integrating over $\mathbb{R}^d$, together with Proposition \ref{a} to get
\begin{equation*}
\begin{split}
&\frac{\mathrm{d}}{\mathrm{d}t}\int_{\mathbb{R}^d}\rho^{q}\, \mathrm{d}x+\frac{4(q-1)}{q}\int_{\mathbb{R}^d}|\nabla\rho^{\frac{q}{2}}|^2\, \mathrm{d}x\\
=&\chi(q-1)\int_{\mathbb{R}^d}\rho^{q+1}\, \mathrm{d}x-\chi(q-1)\int_{\mathbb{R}^d}\rho^{q}f\, \mathrm{d}x\\
\leq&\chi(q-1)\|\rho^{\frac{q}{2}}\|_{L^{\frac{2(d+2\eta)}{d+2\eta-2}}}^2\|\rho\|_{L^{\frac{d}{2}+\eta}}+\chi(q-1)\|\rho\|_{L^q}^q\|f\|_{L^{\infty}}\\
\leq&\chi(q-1)\|\rho^{\frac{q}{2}}\|_{L^2}^{\frac{4\eta}{d+2\eta}}\|\nabla\rho^{\frac{q}{2}}\|_{L^2}^{\frac{2d}{d+2\eta}}+\chi(q-1)\|\rho\|_{L^q}^q\|f\|_{L^{\infty}}\\
\leq&\varepsilon_{1}\|\nabla\rho^{\frac{q}{2}}\|_{L^2}^2+C_{\varepsilon_{1}}(\chi(q-1))^{\frac{d+2\eta}{2\eta}}\|\rho^{\frac{q}{2}}\|_{L^2}^2\|\rho\|_{L^{\frac{d}{2}+\eta}}^{\frac{d+2\eta}{2\eta}}+\chi(q-1)\|\rho\|_{L^q}^q\|f\|_{L^{\infty}}.
\end{split}
\end{equation*}
Let $\varepsilon_{1}=\frac{q-1}{q}$. Using the Gronwall's inequality to the above inequality, together with Proposition \ref{a} again, we have
\begin{equation}\label{bc1000}
\rho\in L^{\infty}(0, T; L^q(\mathbb{R}^d)).
\end{equation}
By the Duhamel's principle, we have
\begin{equation}\label{s30}
	\rho=G\ast\rho_{0}-\chi\int_{0}^{t}G\ast\nabla\cdot(\rho\nabla c)\, \mathrm{d}s,
\end{equation}
where $G$ is the heat kernel. By Young's inequality, Gagliardo-Nirenberg's inequality, Hardy Little-wood inequality and integration by parts, we have
\begin{equation*}
\begin{split}
\|\rho\|_{L^{2q}}\leq&\|G\|_{L^{1}}\|\rho\|_{L^{2q}}+\chi\int_{0}^{t}\|\nabla G\|_{L^{\frac{d(2+d)}{d^2+d+2}}}\|\rho\|_{L^{2q}}\|\nabla c\|_{L^{\frac{d(2+d)}{d-2}}}\, \mathrm{d}s\\
\leq&C+\chi\int_{0}^{t}(t-s)^{-\frac{d}{d+2}}\|\rho\|_{L^{2q}}(\|\rho\|_{L^{\frac{2+d}{2}}}+\|f\|_{L^{\frac{2+d}{2}}})\, \mathrm{d}s.
\end{split}
\end{equation*}
Using the Gronwall's inequality and (\ref{bc1000}), we have
\begin{equation*}
\|\rho\|_{L^{2q}}\leq C+C\chi\int_{0}^{t}(t-s)^{-\frac{d}{d+2}}(1+\|f\|_{L^{1+\frac{d}{2}}})\, \mathrm{d}s\leq C.
\end{equation*}
For the $L^q$-norm of $\nabla\rho$, through (\ref{s30}), we have
\begin{equation}
\begin{split}
\|\nabla\rho\|_{L^q}\leq&\|G\|_{L^1}\|\nabla\rho_{0}\|_{L^q}+\chi\int_{0}^{t}\|\nabla G\ast\nabla\rho\nabla c\|_{L^q}\, \mathrm{d}s+\chi\int_{0}^{t}\|\nabla G\ast\rho\Delta c\|_{L^q}\, \mathrm{d}s\\
\leq&C+\chi\int_{0}^{t}\|\nabla G\|_{L^{\frac{d(2+d)}{d^2+d+2}}}\|\nabla\rho\|_{L^q}\|\nabla c\|_{L^{\frac{d(2+d)}{d-2}}}\, \mathrm{d}s+\chi\int_{0}^{t}\|\nabla G\|_{L^1}\|\rho\|_{L^{2q}}\|\Delta c\|_{L^{2q}}\, \mathrm{d}s\\
\leq&C+\chi\int_{0}^{t}(t-s)^{-\frac{d}{d+2}}\|\nabla\rho\|_{L^q}(\|\rho\|_{L^{\frac{2+d}{2}}}+\|f\|_{L^{\frac{2+d}{2}}})\, \mathrm{d}s\\
&+C\int_{0}^{t}(t-s)^{-\frac{1}{2}}\, \mathrm{d}s+C\int_{0}^{t}(t-s)^{-\frac{r}{2(r-1)}}\, \mathrm{d}s+C\int_{0}^{t}\|f\|_{L^{2q}}^r\, \mathrm{d}s.
\end{split}
\end{equation}
Using the Gronwall's inequality to the above inequality, we finish the proof of this lemma.
\end{proof}

\subsection{The estimates of $\rho^{\varepsilon}$}
In order to get the connection between the particle system \eqref{100} and the limiting PDE, we will introduce a regularized intermediate PDE
\begin{equation}\label{17}
	\partial_t\rho^{\varepsilon}=\Delta\rho^{\varepsilon}-\nabla\cdot(\rho^{\varepsilon}\chi\nabla\widetilde{\Phi}_{\varepsilon}\ast(\rho^{\varepsilon}-f)),
\end{equation} 
and its corresponding Mckean Vlasov stochatic particle system in the next section.
In this subsection, for any given time $0<T<\infty$, we mainly give the estimate for the difference between $\rho$ and $\rho^{\varepsilon}$ on $[0, T]\times\mathbb{R}^d$ and the estimates of $\rho^{\varepsilon}$.

\begin{lemma}\label{e}
	For any $f\in {\mathbb{X}}$, the initial data satisfy (\ref{29}), then we have
	\begin{equation*}
		\|\rho-\rho^{\varepsilon}\|_{L^{\infty}(0, T; L^1(\mathbb{R}^d)\cap L^{\infty}(\mathbb{R}^d))}\leq C\varepsilon.
	\end{equation*}
\end{lemma}

\begin{proof}
	Let $u=\rho-\rho^{\varepsilon}$, it is easy to get
	\begin{align}\label{18}
			\partial_{t}u-\Delta u=&\chi\nabla\cdot(u\nabla\widetilde{\Phi}_{\varepsilon}\ast u)-\chi\nabla\cdot\big(u\nabla\widetilde{\Phi}_{\varepsilon}\ast(\rho-f)\big)
			\nonumber
			\\ 
			&-\chi\nabla\cdot(\rho\nabla\widetilde{\Phi}_{\varepsilon}\ast u)+\chi\nabla\cdot\big(\rho\nabla(\widetilde{\Phi}_{\varepsilon}-\Phi)\ast(\rho-f)\big).
	\end{align}
	Using the Duhamel's principle to (\ref{18}), we have
	\begin{align}\label{19}
			u=&\chi\int_{0}^{t}G\ast\nabla\cdot(u\nabla\widetilde{\Phi}_{\varepsilon}\ast u)\, \mathrm{d}s-\chi\int_{0}^{t}G\ast\nabla\cdot\big(u\nabla\widetilde{\Phi}_{\varepsilon}\ast(\rho-f)\big)\, \mathrm{d}s\nonumber
			\\
			&-\chi\int_{0}^{t}G\ast\nabla\cdot(\rho\nabla\widetilde{\Phi}_{\varepsilon}\ast u)\, \mathrm{d}s+\chi\int_{0}^{t}G\ast\nabla\cdot\big(\rho\nabla(\widetilde{\Phi}_{\varepsilon}-\Phi)\ast(\rho-f)\big)\, \mathrm{d}s\nonumber
			\\
			:=&O_{1}+O_{2}+O_{3}+O_{4},
	\end{align}
	where $G$ is the heat kernel. For the terms $O_{1}, O_{2}$ and $O_{3}$, using the Young's inequality and Lemma \ref{phi1}, we have
	\begin{align*}
			\|O_{1}\|_{L^p}\leq&\int_{0}^{t}\|\nabla G\ast(u\nabla\widetilde{\Phi}_{\varepsilon}\ast u)\|_{L^p}\, \mathrm{d}s\\
			\leq&\int_{0}^{t}\|\nabla G\|_{L^1}\|u\|_{L^p}\|\nabla\widetilde{\Phi}_{\varepsilon}\ast u\|_{L^{\infty}}\, \mathrm{d}s\\
			\leq&C(d)\int_{0}^{t}(t-s)^{-\frac{1}{2}}\|u\|_{L^p}(\|u\|_{L^1}+\|u\|_{L^{\infty}})\, \mathrm{d}s,
	\end{align*}
	\begin{align*}
			\|O_{2}\|_{L^p}\leq&\int_{0}^{t}\|\nabla G\ast\big(u\nabla\widetilde{\Phi}_{\varepsilon}\ast(\rho-f)\big)\|_{L^p}\, \mathrm{d}s\\
			\leq&\int_{0}^{t}\|\nabla G\|_{L^1}\|u\|_{L^p}\|\nabla\widetilde{\Phi}_{\varepsilon}\ast(\rho-f)\|_{L^{\infty}}\, \mathrm{d}s\\
			\leq&C(d)\int_{0}^{t}(t-s)^{-\frac{1}{2}}\|u\|_{L^p}(\|\rho-f\|_{L^1}+\|\rho-f\|_{L^{\infty}})\, \mathrm{d}s,
	\end{align*}
	and
		\begin{align*}
			\|O_{3}\|_{L^p}\leq&\int_{0}^{t}\|\nabla G\ast(\rho\nabla\widetilde{\Phi}_{\varepsilon}\ast u)\|_{L^p}\, \mathrm{d}s\\
			\leq&\int_{0}^{t}\|\nabla G\|_{L^1}\|\rho\|_{L^p}\|\nabla\widetilde{\Phi}_{\varepsilon}\ast u\|_{L^{\infty}}\, \mathrm{d}s\\
			\leq&C(d)\int_{0}^{t}(t-s)^{-\frac{1}{2}}\|\rho\|_{L^p}(\|u\|_{L^1}+\|u\|_{L^{\infty}})\, \mathrm{d}s.
	\end{align*}
	For the term $O_{4}$, by Lemma \ref{f} and \ref{g}, we have
		\begin{align*}
			\|O_{4}\|_{L^p}\leq&\int_{0}^{t}\|\nabla G\ast\big(\rho\nabla(\widetilde{\Phi}_{\varepsilon}-\Phi_{\varepsilon})\ast(\rho-f)\big)\|_{L^p}\, \mathrm{d}s\\
			&+\int_{0}^{t}\|\nabla G\ast\big(\rho\nabla(\Phi_{\varepsilon}-\Phi)\ast(\rho-f)\big)\|_{L^p}\, \mathrm{d}s\\
			\leq&\int_{0}^{t}\|\nabla G\|_{L^1}\|\rho\|_{L^p}\|j_{\varepsilon}\ast\nabla(\widetilde{\Phi}-\Phi)\ast(\rho-f)\|_{L^{\infty}}\, \mathrm{d}s\\
			&+\int_{0}^{t}\|\nabla G\|_{L^1}\|\rho\|_{L^p}\|\nabla(\Phi_{\varepsilon}-\Phi)\ast(\rho-f)\|_{L^{\infty}}\, \mathrm{d}s\\
			\leq&C(d)\varepsilon\int_{0}^{t}(t-s)^{-\frac{1}{2}}\|\rho\|_{L^p}\|\rho-f\|_{L^{\infty}}\, \mathrm{d}s\\
			&+C\varepsilon\int_{0}^{t}(t-s)^{-\frac{1}{2}}\|\rho\|_{L^p}\|\rho-f\|_{W^{1, q}}\, \mathrm{d}s.
		\end{align*}
Let $p=1$ and $\infty$, notice that $u(0,\cdot)=\rho_0-\rho_0=0$, by putting the above estimates into the resulting inequality, we have
\begin{equation*}
	\begin{split}
	\|u\|_{L^1}+\|u\|_{L^{\infty}}\leq&C(d)\int_{0}^{t}(t-s)^{-\frac{1}{2}}(1+\|u\|_{L^1\cap L^{\infty}}+\|\rho\|_{L^1\cap L^{\infty}}+\|f\|_{L^1\cap L^{\infty}})(\|u\|_{L^1}+\|u\|_{L^{\infty}})\, \mathrm{d}s\\
	&+C(d)\varepsilon\int_{0}^{t}(t-s)^{-\frac{1}{2}}\|\rho\|_{L^1\cap L^{\infty}}(\|\rho\|_{W^{1, q}}+\|f\|_{W^{1, q}})\, \mathrm{d}s.
	\end{split}
\end{equation*}
Using the Gronwall's inequality to the above inequality, together with Proposition \ref{bca}, we complete the proof of this lemma.
\end{proof}

\begin{lemma}\label{b10}
For any $f\in\mathbb{X}$, the initial data satisfy (\ref{29}), then we have
\begin{equation*}
\nabla\rho^{\varepsilon}\in L^{\infty}(0, T; L^q(\mathbb{R}^d)).
\end{equation*}
\end{lemma}

\begin{proof}
According to the equation (\ref{17}), we have
\begin{equation}\label{b100}
\rho^{\varepsilon}=G\ast\rho^{\varepsilon}_{0}-\chi\int_{0}^{t}G\ast\nabla\cdot\big(\rho^{\varepsilon}\nabla\widetilde{\Phi}_{\varepsilon}\ast(\rho^{\varepsilon}-f)\big)\, \mathrm{d}s,
\end{equation}
where $G$ is the heat kernel. In order to get the $L^q$- norm of $\nabla\rho^{\varepsilon}$, we need the following estimates which obtained mainly by Lemma \ref{e}, \ref{phi1} and \ref{bcyl1}.
\begin{equation*}
\begin{split}
&\int_{0}^{t}\|\nabla G\ast(\nabla\rho^{\varepsilon}\nabla\widetilde{\Phi}_{\varepsilon}\ast\big(\rho^{\varepsilon}-f)\big)\|_{L^q}\, \mathrm{d}s\\
\leq&\int_{0}^{t}\|\nabla G\|_{L^1}\|\nabla\rho^{\varepsilon}\|_{L^q}\|\nabla\widetilde{\Phi}_{\varepsilon}\ast(\rho^{\varepsilon}-f)\|_{L^{\infty}}\, \mathrm{d}s\\
\leq&C(d)\int_{0}^{t}(t-s)^{-\frac{1}{2}}\|\nabla\rho^{\varepsilon}\|_{L^q}\|\rho^{\varepsilon}-f\|_{L^{1}\cap L^{\infty}}\, \mathrm{d}s
\end{split}
\end{equation*}
and
\begin{equation*}
\begin{split}
&\int_{0}^{t}\|\nabla G\ast\big(\rho^{\varepsilon}\nabla^2\widetilde{\Phi}_{\varepsilon}\ast(\rho^{\varepsilon}-f)\big)\|_{L^q}\, \mathrm{d}s\\
\leq&\int_{0}^{t}\|\nabla G\|_{L^1}\|\rho^{\varepsilon}\|_{L^q}\|\nabla^2\widetilde{\Phi}_{\varepsilon}\ast(\rho^{\varepsilon}-f)\|_{L^{\infty}}\, \mathrm{d}s\\
\leq&C(d)\int_{0}^{t}\|\nabla G\|_{L^1}\|\rho^{\varepsilon}\|_{L^q}\|\rho^{\varepsilon}-f\|_{L^1\cap W^{1, q}}\, \mathrm{d}s\\
\leq&C(d)\int_{0}^{t}\|\nabla G\|_{L^1}\, \mathrm{d}s+C(d)\Big(\int_{0}^{t}\|\nabla G\|_{L^1}^{\frac{r}{r-1}}\, \mathrm{d}s\Big)^{1-\frac{1}{r}}\Big(\int_{0}^{t}\|f\|_{L^1\cap W^{1, q}}^r\, \mathrm{d}s\Big)^{\frac{1}{r}}\\
&+C(d)\int_{0}^{t}\|\nabla G\|_{L^1}\|\nabla\rho^{\varepsilon}\|_{L^q}\, \mathrm{d}s.
\end{split}
\end{equation*}
Taking the $L^q$-norm of (\ref{b100}) operated by $\nabla$, then putting the above inequalities into the resulting inequality, using the Gronwall's inequality, we finish the proof of this lemma.
\end{proof}

\section[Propagation of chaos]{Propagation of chaos for fixed control function}\label{meanfield}
In this section, we present the mean-field limit result for given control function $f\in {\mathbb{X}}$. We will give the convergence in probability for trajectory differences and the strong $L^1$ convergence of the $1$-st marginal of the $N$-particle distribution to the PDE solution. Both results are going to be used in the last section to prove the main theorem.

Following the framework of moderate interacting concept, we introduce the intermediate Mckean-Vlasov particle system
\begin{equation}\label{interparticle}
	\begin{cases}
		\mathrm{d}\overline{X}_{i}^{\varepsilon}=\chi\nabla\widetilde{\Phi}_{\varepsilon} \ast\rho^\varepsilon
		(t,\overline{X}_{i}^{\varepsilon})\,  \mathrm{d}t-\chi\nabla\widetilde{\Phi}_{\varepsilon}\ast f(t,\overline{X}_{i}^{\varepsilon})\,  \mathrm{d}t+\sqrt{2}\, \mathrm{d}W_{t}^{i},\\
		\overline{X}_{i}^{\varepsilon}(0)=\xi_{i}.
	\end{cases}
\end{equation}
For any fixed $\varepsilon$, it is well-known that \eqref{interparticle} has a unique square integrable solution $\overline{X}_i^\varepsilon[f]$ and its distribution $\rho^\varepsilon[f]$ is a weak solution of \eqref{17}, we refer for example to \cite[Chapter 3]{LackerNotes} as a reference.

\subsection{Convergence in probability}

We first recall the law of large number result from \cite{LCZ,CHJ24}.

\begin{lemma}\label{j}
	Assume $(\overline{Y}_{i})_{i=1,\cdots,N}$ are independent and identically distributed $\mathbb{R}^d$ valued random variables with commom distribution $v\in L^1(\mathbb{R}^d)$, $U\in L^\infty(\mathbb{R}^d)$ is given and for any  $\theta\in \mathbb{R}$
	\begin{equation*}
		\mathcal{B}_{\theta}^{i}(U, v)=\Big\{\omega\in\Omega: \Big|\frac{1}{N}\sum_{j=1}^{N}U(\overline{Y}^{i}-\overline{Y}^{j})-U\ast v(\overline{Y}^i)\Big|>\frac{1}{N^{\theta}}\Big\}.
	\end{equation*}
	Let $\mathcal{B}_{\theta}^{N}(U, v)=\bigcup_{i=1}^{N}\mathcal{B}_{\theta}^{i}(U, v)$, then for any $\widetilde{k}\geq 1$, we have
	\begin{equation*}
		\mathbb{P}(\mathcal{B}_{\theta}^{N}(U, v))\leq N\max_{1\leq i\leq N}\mathbb{P}(\mathcal{B}_{\theta}^i(U, v))\leq C(\tilde{k}) N^{2\tilde{k}(\theta-\frac{1}{2})+1}\|U\|_{L^{\infty}(\mathbb{R}^d)}^{2\tilde{k}}.
	\end{equation*}
\end{lemma}

The first propagation of chaos is the following convergence  $X^{N,\varepsilon}_i[f]-\overline{X}^\varepsilon_i[f]\rightarrow 0$ in the sense of probability, namely
\begin{proposition}\label{Convergenceinprobability}
	For given $f\in {\mathbb{X}}$, suppose $X^{N,\varepsilon}_i[f]$ and $\overline{X}^\varepsilon_i[f]$ are the solutions of \eqref{100} and \eqref{interparticle} separately. For $\alpha\in (0,\frac12)$, let 
	\begin{equation}\label{Adef}
		 \mathcal{A}_{\alpha}=\Big\{\omega\in\Omega: \max_{1\leq i\leq N} \Big|X^{N, \varepsilon}_i[f]-\overline{X}^{\varepsilon}_i[f]\Big|\geq N^{-\alpha}\Big\},
	\end{equation}
then there exists $\beta_*>0$, such that for $\beta\in (0,\beta_*)$ it holds for $\varepsilon=N^{-\beta}$
\begin{equation}
	\label{ProA}\mathbb{P}(\mathcal{A}_\alpha)\leq \frac{C(\gamma)}{N^\gamma}, \quad \forall\gamma>0.
\end{equation}
\end{proposition}

\begin{proof} In the proof, we omit the dependence of $f$ for convenience. Namely we denote $X^{N,\varepsilon}_i[f]=X^{N,\varepsilon}_i$ and $\overline{X}_i^{\varepsilon}[f]=\overline{X}_i^{\varepsilon}$.
	
 We follow the method given in \cite{LCZ,CHJ24}. By introducing the following stopping time
	\begin{equation*}
		\tau(\omega)=\inf\Big\{t\in[0, T]|\max_{1\leq i\leq N}\Big|X_{i}^{N, \varepsilon}(t)-\overline{X}_{i}^{\varepsilon}(t)\Big|\geq N^{-\alpha}\Big\},
	\end{equation*}
we define the stopped process
	\begin{equation*}
		S(t)=N^{2\alpha k}\max_{1\leq i\leq N}\Big|(X_{i}^{N, \varepsilon}-\overline{X}_{i}^{\varepsilon})(t\wedge\tau)\Big|^{2k}\leq 1.
	\end{equation*}
Then Markov's inequality implies that it is enough to estimate $\sup_{0\leq t\leq T}\mathbb{E}(S(t))$, i.e.
	\begin{equation*}
		\sup_{0\leq t\leq T}\mathbb{P}\Big(\max_{1\leq i\leq N}\Big|X_{i}^{N, \varepsilon}-\overline{X}_{i}^{\varepsilon}\Big|\geq N^{-\alpha}\Big)\leq\sup_{0\leq t\leq T}\mathbb{E}(S(t)=1)\leq\sup_{0\leq t\leq T}\mathbb{E}(S(t)).
	\end{equation*}
Thus for any $t\land\tau$, we have
	\begin{equation*}
		\begin{split}
			\mathbb{E}(S(t))=&N^{2\alpha k}\mathbb{E}\big(\max_{1\leq i\leq N}|(X_{i}^{N, \varepsilon}-\overline{X}_{i}^{\varepsilon})(t\land\tau)|^{2k}\big)\\
			\leq&2k \chi N^{\alpha}\mathbb{E}\Big(\max_{1\leq i\leq N}\int_{0}^{t\land\tau}\Big|(\nabla\widetilde{\Phi}_{\varepsilon}\ast\mu_{N})(X_{i}^{N, \varepsilon}(s))-(\nabla\widetilde{\Phi}_{\varepsilon}\ast\overline{\mu}_{N})(\overline{X}_{i}^{\varepsilon}(s))\Big|S(s)^{\frac{2k-1}{2k}} \, \mathrm{d}s\Big)\\
			&+2k\chi N^{\alpha}\mathbb{E}\Big(\max_{1\leq i\leq N}\int_{0}^{t\land\tau}\Big|(\nabla\widetilde{\Phi}_{\varepsilon}\ast\overline{\mu}_{N})(\overline{X}_{i}^{\varepsilon}(s))-(\nabla\widetilde{\Phi}_{\varepsilon}\ast\rho^{\varepsilon})(s,\overline{X}_{i}^{\varepsilon}(s))\Big| S(s)^{\frac{2k-1}{2k}}\, \mathrm{d}s\Big)\\
			&+2k\chi N^{\alpha}\mathbb{E}\Big(\max_{1\leq i\leq N}\int_{0}^{t\land\tau}\Big|(\nabla\widetilde{\Phi}_{\varepsilon}\ast f)(s,{X}_{i}^{N, \varepsilon}(s))-(\nabla\widetilde{\Phi}_{\varepsilon}\ast f)(s,\overline{X}_{i}^{\varepsilon}(s))\Big| S(s)^{\frac{2k-1}{2k}}\, \mathrm{d}s\Big)\\
			:=&K_{1}+K_{2}+K_{3},
		\end{split}
	\end{equation*}
where $\overline{\mu}_N$ is the empirical measure of $\overline{X}_{i}^{\varepsilon}$, $i=1,\cdots,N$.

For $K_{1}$ we use the Taylor's expansion and obtain that
	\begin{equation*}
		\begin{split}
			K_{1}\leq &\, 2k\chi N^{\alpha}\mathbb{E}\Big(\max_{1\leq i\leq N}\int_{0}^{t\land\tau}\Big|\frac{1}{N}\sum_{j=1}^{N}\nabla^2\widetilde{\Phi}_{\varepsilon}(\overline{X}_{i}^{\varepsilon}-\overline{X}_{j}^{\varepsilon})(X_{i}^{N, \varepsilon}-X_{j}^{N, \varepsilon}-(\overline{X}_{i}^{\varepsilon}-\overline{X}_{j}^{\varepsilon}))\Big| S(s)^{\frac{2k-1}{2k}}\, \mathrm{d}s\Big)\\
			&+2k\chi N^{\alpha}\mathbb{E}\Big(\max_{1\leq i\leq N}\int_{0}^{t\land\tau}\|\nabla^3\widetilde{\Phi}_{\varepsilon}\|_{L^{\infty}}\Big|\frac{1}{N}\sum_{j=1}^{N}(X_{i}^{N, \varepsilon}-X_{j}^{N, \varepsilon}-(\overline{X}_{i}^{\varepsilon}-\overline{X}_{j}^{\varepsilon}))\Big|^2 S(s)^{\frac{2k-1}{2k}}\, \mathrm{d}s\Big)\\
			\leq&K_{11}+\frac{2k\chi }{N^{\alpha}}\|\nabla^3\widetilde{\Phi}_{\varepsilon}\|_{L^{\infty}}\int_{0}^{t}\mathbb{E}(S(s))\, \mathrm{d}s.
		\end{split}
	\end{equation*}
For $K_{11}$ we use the law of large number to approximate the coefficient by a convolution form, namely
\begin{align*}
			K_{11}\leq&4k\chi N^{\alpha}\mathbb{E}\Big(\max_{1\leq i\leq N}\int_{0}^{t\land\tau}\frac{1}{N}\sum_{j=1}^{N}|\nabla^2\widetilde{\Phi}_{\varepsilon}(\overline{X}_{i}^{\varepsilon}-\overline{X}_{j}^{\varepsilon})| S(s)\, \mathrm{d}s\Big)\\
			\leq&4k\chi\mathbb{E}\Big(\max_{1\leq i\leq N}\int_{0}^{t\land\tau}\Big(\frac{1}{N}\sum_{j=1}^{N}|\nabla^2\widetilde{\Phi}_{\varepsilon}(\overline{X}_{i}^{\varepsilon}-\overline{X}_{j}^{\varepsilon})|-|\nabla^2\widetilde{\Phi}_{\varepsilon}|\ast\rho^{\varepsilon}(s,\overline{X}_{i}^{\varepsilon})\Big)S(s)\, \mathrm{d}s\Big)\\
			&+4k\chi\mathbb{E}\Big(\max_{1\leq i\leq N}\int_{0}^{t\land\tau}|\nabla^2\widetilde{\Phi}_{\varepsilon}|\ast\rho^{\varepsilon}(s,\overline{X}_{i}^{\varepsilon})S(s)\, \mathrm{d}s\Big)\\
			:=&K_{111}+K_{112},
	\end{align*}
where the second term can be easily estimated by the definition of $\widetilde{\Phi}$,
	\begin{equation*}
		K_{112}\leq 4k\chi (1-\ln 2\varepsilon)\int_{0}^{t}\mathbb{E}(S(s))\, \mathrm{d}s.
	\end{equation*}

For $K_{111}$, by taking $U=|\nabla^2\widetilde{\Phi}_{\varepsilon}|,\ v=\rho^{\varepsilon}$ and $\overline{Y}_i=\overline{X}_i^\varepsilon$ in Lemma \ref{j}, then for $\theta\in [0,\frac12)$ we have
	\begin{equation*}
		\begin{split}
			K_{111}\leq&4k\chi\mathbb{E}\Big(\max_{1\leq i\leq N}\int_{0}^{t\land\tau}({\mathbb{I}}_{(\mathcal{B}^N_{\theta})^c}+{\mathbb{I}}_{\mathcal{B}^N_{\theta}})\Big(\frac{1}{N}\sum_{j=1}^{N}|\nabla^2\widetilde{\Phi}_{\varepsilon}(\overline{X}_{i}^{\varepsilon}-\overline{X}_{j}^{\varepsilon})|-|\nabla^2\widetilde{\Phi}_\varepsilon|\ast\rho^{\varepsilon}(s,\overline{X}_i^\varepsilon) \Big)S(s)\, \mathrm{d}s\Big)\\
			\leq&4k\chi\int_{0}^{t}\mathbb{E}(S(s))\, \mathrm{d}s+C(t, k, k_1)\|\nabla^2\widetilde{\Phi}_{\varepsilon}\|_{L^{\infty}(\mathbb{R}^d)}^{1+2 k_1}N^{2 k_1(\theta-\frac{1}{2})+1}.
		\end{split}
	\end{equation*}
For the estimate of $K_2$ we  use the Young's inequality and again the law of large number in Lemma \ref{j} with  $U=\nabla \widetilde{\Phi}_{\varepsilon},\ v=\rho^{\varepsilon}$ and $\overline{Y}_i=\overline{X}_i^\varepsilon$, and obtain that for $\theta_1\in (0,\frac12)$
	\begin{equation*}
		\begin{split}
			K_{2}\leq&C(k)N^{2\alpha k}\mathbb{E}\Big(\max_{1\leq i\leq N}\int_{0}^{t\land\tau}\Big|(\nabla\widetilde{\Phi}_{\varepsilon}\ast\overline{\mu}_{N})(\overline{X}_{i}^{\varepsilon}(s))-(\nabla\widetilde{\Phi}_{\varepsilon}\ast\rho^{\varepsilon})(s, \overline{X}_{i}^{\varepsilon}(s))\Big|^{2k}\, \mathrm{d}s\Big)+\int_{0}^{t}\mathbb{E}(S(s))\, \mathrm{d}s\\
			\leq&C(k)N^{2k(\alpha-\theta_{1})}+C(t, k_2)N^{2\alpha k}\|\nabla\widetilde{\Phi}_{\varepsilon}\|_{L^{\infty}(\mathbb{R}^d)}^{2(k+k_{2})}N^{2k_2(\theta_{1}-\frac{1}{2})+1}+\int_{0}^{t}\mathbb{E}(S(s))\, \mathrm{d}s.
		\end{split}
	\end{equation*}
Furthermore, for the estimate of $K_{3}$, by Lemma \ref{bcyl1}
	\begin{equation*}
			K_{3}\leq 2k\chi\int_{0}^{t}l(s)\mathbb{E}(S(s))\, ds.
	\end{equation*}
As a summary, we obtain
	\begin{equation*}
		\begin{split}
			\mathbb{E}(S(t))\leq&4k\chi\int_{0}^{t}\Big(1+N^{-\alpha}\|\nabla^3\widetilde{\Phi}_{\varepsilon}\|_{L^{\infty}(\mathbb{R}^d)}-\ln 2\varepsilon+l(s)\Big)\mathbb{E}(S(s))\, \mathrm{d}s\\
			&+C(k)N^{(\alpha-\theta_{1})2k}+C(t, k, k_1)N^{2 k_1(\theta-\frac{1}{2})+1}\|\nabla^2\widetilde{\Phi}_{\varepsilon}\|_{L^{\infty}(\mathbb{R}^d)}^{2 k_1+1}\\
			&+C(t, k, k_2)N^{2\alpha k+2 k_2(\theta_{1}-\frac{1}{2})+1}\|\nabla\widetilde{\Phi}_{\varepsilon}\|_{L^{\infty}(\mathbb{R}^d)}^{2(k+k_2)}.
		\end{split}
	\end{equation*}
By putting into the estimates $\|\nabla\widetilde{\Phi}_{\varepsilon}\|_{L^{\infty}(\mathbb{R}^d)}\leq\varepsilon^{1-d},\ \|\nabla^2\widetilde{\Phi}_{\varepsilon}\|_{L^{\infty}(\mathbb{R}^d)}\leq\varepsilon^{-d},\ \|\nabla^3\widetilde{\Phi}_{\varepsilon}\|_{L^{\infty}(\mathbb{R}^d)}\leq\varepsilon^{-(d+1)}$ and $\theta=0$, $\varepsilon=N^{-\beta}$, we arrive at
	\begin{equation*}
		\begin{split}
			\mathbb{E}(S(t))\leq&4k\chi\int_{0}^{t}\Big(1+N^{-\alpha+\beta(d+1)}-\ln (2N^{-\beta})+l(s)\Big)\mathbb{E}(S(s))\, \mathrm{d}s\\
			&+C(k)N^{(\alpha-\theta_{1})2k}+C(t, k, k_1)N^{k_1(2\beta d-1)+\beta d+1}\\
			&+C(t, k, k_2)N^{2k(\alpha+\beta (d-1))+1 +2k_2((\theta_{1}-\frac{1}{2})+\beta(d-1))}.
		\end{split}
	\end{equation*}
	In the end, under the condition
	\begin{equation}\label{conditiontheta}
	\alpha<\theta_1,\quad 2\beta d<1,\quad \theta_1<\frac12-\beta(d-1), 
	\end{equation}
we have that for any $\gamma>0$ there exists $k=\frac{\gamma}{\theta_1-\alpha}$ such that $N^{(\alpha-\theta_1)2k}= N^{-2\gamma}$, and there exists $k_1$ such that 
$$N^{k_1(2\beta d-1)+\beta d+1}\leq N^{-2\gamma}.$$
Furthermore, for this selected $k$, there exists $k_2$ such that
$$
N^{2k(\alpha+\beta (d-1))+1 +2k_2((\theta_{1}-\frac{1}{2})+\beta(d-1))}\leq N^{-2\gamma}.
$$
Then if $\alpha\geq\beta (d+1)$, i.e. $N^{-\alpha+\beta(d+1)}\leq 1$, Gronwall's inequality implies that if $\beta\leq \frac{\theta_1-\alpha}{4\chi t}$ then
\begin{equation*}
\mathbb{E}(S(t))\leq C(\gamma)N^{-2\gamma}\exp\Big\{\frac{4\gamma\chi}{\theta_1-\alpha}\int_{0}^{t}(1+\ln N^{\beta}+l(s))\, \mathrm{d}s\Big\}\leq C(\gamma)N^{-2\gamma} N^{\frac{4\beta\chi}{\theta_1-\alpha}\gamma t}\leq C(\gamma)N^{-\gamma}.
\end{equation*}
The above process requires that $0<\beta<\min\{\frac{1}{2d}, \frac{\alpha}{d+1}, \frac{\frac12-\theta_{1}}{d-1},\frac{\theta_{1}-\alpha}{4\chi t}\}:=\beta_*$.
\end{proof}

\subsection{Strong $L^1$ convergence}
The second propagation of chaos result, which will be used in the proof of main theorem, follows the idea of combination of relative entropy and convergence in probability given in \cite{ChenHolzingerHuo2023}.
\begin{proposition}\label{k}
	Assume that the assumptions in Proposition \ref{Convergenceinprobability} hold. 
 Let $\rho^{N,\varepsilon}$ be the solution of \eqref{22}, which is the joint distribution of $X^{N,\varepsilon}_i$, $i=1,\cdots,N$, and $\rho^{N,\varepsilon;1}$ be its one particle marginal. Let $\rho$ be the solution of \eqref{1}, then we have
 \begin{equation}
 	\label{L1convergence} 
 	\|\rho^{N,\varepsilon;1}-\rho\|_{L^\infty(0,T;L^1(\mathbb{R}^d))}\leq\frac{C(t)}{N^\beta}, \ where\  0<\beta<\beta_*.
 \end{equation}
\end{proposition}

\begin{proof} By using super additivity property of relative entropy and Csisz\'ar-Kullback-Pinsker inequality, given in the attachment of \cite{ChenHolzingerHuo2023}, the estimate of \eqref{L1convergence} can be reduced in doing the following relative entropy estimate
	\begin{equation}\label{estrelativeentropy}
		\mathcal{H}(\rho^{N, \varepsilon}| \rho^{\otimes N}):=\frac{1}{N}\int_{\mathbb{R}^{dN}}\log \Big(\frac{\rho^{N, \varepsilon}}{\rho^{\otimes N}}\Big)\rho^{N, \varepsilon}\,  \mathrm{d}x_{1}\cdots\,  \mathrm{d}x_{N}\leq\frac{C(t)}{N^{2\beta}}.
	\end{equation}
By direct computation and using Young's inequality, we obtain
	\begin{align*}
			&\frac{ \mathrm{d}}{ \mathrm{d}t}\mathcal{H}(\rho^{N, \varepsilon}| \rho^{\otimes N})+\frac{1}{2N}\int_{\mathbb{R}^{dN}}\sum_{i=1}^{N}\Big|\nabla_{x_{i}}\log \Big(\frac{\rho^{N, \varepsilon}}{\rho^{\otimes N}}\Big)\Big|^2\rho^{N, \varepsilon}\,  \mathrm{d}x_{1}\cdots\,  \mathrm{d}x_{N}\\
			\leq&\chi\mathbb{E}\Big(\frac{1}{N}\sum_{i=1}^{N}\Big|\nabla\Phi\ast\rho(t, X_{i}^{N, \varepsilon})-\frac{1}{N}\sum_{j=1}^{N}\nabla\widetilde{\Phi}_{\varepsilon}(X_{i}^{N, \varepsilon}-X_{j}^{N, \varepsilon})\Big|^2\Big)\\
			&\qquad +\chi\mathbb{E}\Big(\frac{1}{N}\sum_{i=1}^{N}\big|\nabla\Phi\ast f(t,X_{i}^{N, \varepsilon})-\nabla\widetilde{\Phi}_{\varepsilon}\ast f(t,X_{i}^{N, \varepsilon})\big|^2\Big)\\
			:=&\chi M_{1}+\chi M_{2}.
		\end{align*}
The estimate for $M_2$ can be directly obtained by the error estimate $\|\nabla\Phi\ast f-\nabla\widetilde{\Phi}_\varepsilon\ast f\|_{L^\infty(\mathbb{R}^d)}$, namely, by using Lemma \ref{f} and \ref{g} we have
\begin{align*}
	M_2\leq \|\nabla\Phi\ast f-\nabla\widetilde{\Phi}_\varepsilon\ast f\|_{L^\infty(\mathbb{R}^d)}^2\leq \varepsilon^2 \|f\|_{W^{1,q}(\mathbb{R}^d)}^2\leq \varepsilon^2 l^2(t).
\end{align*}
Term $M_{1}$ reflects the mean-field limit estimate, where we need the convergence in probability result in Proposition \ref{Convergenceinprobability}. We insert the intermediate dynamic $\overline{X}_{i}^{\varepsilon}$ and use the i.i.d. property of it to proceed the following estimate 
\begin{align*}
	M_{1}\leq&\mathbb{E}\Big(\frac{1}{N}\sum_{i=1}^{N}\Big|\nabla\Phi\ast\rho(t, X_{i}^{N, \varepsilon})-\nabla\widetilde{\Phi}_{\varepsilon}\ast\rho^{\varepsilon}(t, X_{i}^{N, \varepsilon})\Big|^2\Big)\\
	&+\mathbb{E}\Big(\frac{1}{N}\sum_{i=1}^{N}\Big|\nabla\widetilde{\Phi}_{\varepsilon}\ast\rho^{\varepsilon}(t, X_{i}^{N, \varepsilon})-\nabla\widetilde{\Phi}_{\varepsilon}\ast\rho^{\varepsilon}(t, \overline{X}_{i}^{\varepsilon})\Big|^2\Big)\\
	&+\mathbb{E}\Big(\frac{1}{N}\sum_{i=1}^{N}\Big|\nabla\widetilde{\Phi}_{\varepsilon}\ast\rho^{\varepsilon}(t, \overline{X}_{i}^{\varepsilon})-\frac{1}{N}\sum_{j=1}^{N}\nabla\widetilde{\Phi}_{\varepsilon}(\overline{X}_{i}^{\varepsilon}-\overline{X}_{j}^{\varepsilon})\Big|^2\Big)\\
	&+\mathbb{E}\Big(\frac{1}{N}\sum_{i=1}^{N}\Big|\frac{1}{N}\sum_{j=1}^{N}\big(\nabla\widetilde{\Phi}_{\varepsilon}(\overline{X}_{i}^{\varepsilon}-\overline{X}_{j}^{\varepsilon})-\nabla\widetilde{\Phi}_{\varepsilon}(X_{i}^{N, \varepsilon}-X_{j}^{N, \varepsilon})\big)\Big|^2\Big):=\sum_{i=1}^{4}M_{1i}.
\end{align*}
These four terms have typical structures, $M_{11}$ can be estimated by pure PDE analysis, $M_{12}$ and $M_{14}$ share the structure of mean-field limit, and $M_{13}$ will be estimated by the law of large number since it is only for i.i.d. random variables. In the following we discuss them separately.

For $M_{11}$, by using Lemma \ref{f}, \ref{phi1} and \ref{g}, and the estimate for $\rho-\rho^\varepsilon$ in Lemma \ref{e}, we have
\begin{align*}
			M_{11}\leq&\mathbb{E}\Big(\frac{1}{N}\sum_{i=1}^{N}\Big|\nabla\Phi\ast\rho(t, X_{i}^{N, \varepsilon})-\nabla\Phi_{\varepsilon}\ast\rho (t, X_{i}^{N, \varepsilon})\Big|^2\Big)\\
			&+\mathbb{E}\Big(\frac{1}{N}\sum_{i=1}^{N}\Big|\nabla\Phi_{\varepsilon}\ast\rho(t, X_{i}^{N, \varepsilon})-\nabla\widetilde{\Phi}_{\varepsilon}\ast\rho(t, X_{i}^{N, \varepsilon})\Big|^2\Big)\\
			&+\mathbb{E}\Big(\frac{1}{N}\sum_{i=1}^{N}\Big|\nabla\widetilde{\Phi}_{\varepsilon}\ast\rho(t, X_{i}^{N, \varepsilon})-\nabla\widetilde{\Phi}_{\varepsilon}\ast\rho^{\varepsilon}(t, X_{i}^{N, \varepsilon})\Big|^2\Big)\\
			\leq & \|\nabla(\Phi-\Phi_{\varepsilon})\ast\rho\|^2_{L^\infty(\mathbb{R}^d)} + \|\nabla(\Phi_{\varepsilon}-\widetilde{\Phi}_\varepsilon)\ast\rho\|^2_{L^\infty(\mathbb{R}^d)}+\|\nabla\widetilde{\Phi}_{\varepsilon}\ast(\rho-\rho^\varepsilon)\|^2_{L^\infty(\mathbb{R}^d)} \\
			\leq &2\varepsilon^2 \|\rho\|^2_{L^\infty(0,T;W^{1,q}(\mathbb{R}^d))}+\|\rho-\rho^{\varepsilon}\|^2_{L^\infty(0,T;L^1(\mathbb{R}^d)\cap L^\infty(\mathbb{R}^d))}\leq C\varepsilon^2=\frac{C}{N^{2\beta}}.
	\end{align*}
The estimates for mean-field limit terms $M_{12}$ and $M_{14}$ need application of the result in Proposition \ref{Convergenceinprobability}. By splitting $\Omega=\mathcal{A}_\alpha\cup \mathcal{A}_\alpha^c$, we have that for $\beta\in (0,\beta_*)$ and $\gamma=2(\alpha-\beta)$ it holds
	\begin{align*}
			M_{12}+M_{14}\leq&C\frac{\|\nabla^2\widetilde{\Phi}_{\varepsilon}\|_{L^{\infty}(\mathbb{R}^d)}^2}{N^{2\alpha}}+C(\gamma)\frac{\|\nabla\widetilde{\Phi}_{\varepsilon}\|_{L^{\infty}(\mathbb{R}^d)}^2}{N^{\gamma}}\leq C(N^{2(\beta d-\alpha)}+N^{2\beta(d-1)-\gamma})=N^{2(\beta d-\alpha)}.
	\end{align*}
The term $M_{13}$ can be estimated by directly applying the law of large number, namely
	\begin{equation*}
		M_{13}\leq C\frac{\|\nabla\widetilde{\Phi}_{\varepsilon}\|_{L^{\infty}(\mathbb{R}^d)}^2}{N}\leq C\frac{\varepsilon^{2(1-d)}}{N}=CN^{2\beta(d-1)-1}.
	\end{equation*}
Therefore, for the relative entropy estimate we obtain
	\begin{equation*}
		\frac{\mathrm{d}}{\mathrm{d}t}\mathcal{H}(\rho^{N, \varepsilon}| \rho^{\otimes N})
			\leq C\chi N^{-2\beta} (l^2(t)+1)+C\chi (N^{2(\beta d-\alpha)}+N^{2\beta(d-1)-1}).
	\end{equation*}
Hence considering the fact that for $\beta_*\leq\frac{\alpha}{d+1}$ and $\beta_{*}\leq\frac{1}{2d}$, we arrive at the result by integrating in time, i.e.
$$
N^{2(\beta d-\alpha)}+N^{2\beta(d-1)-1}\leq N^{-2\beta}.
$$
\end{proof}

\section{The proof of Theorem \ref{mainthm}} \label{proofmainthm}

This section is devoted to prove the main theorem. We first give a list of convergence results for cost functional on both micro- and macro- levels, which help in completing the $\Gamma$-convergence of cost functional.
\begin{lemma}\label{lemmaJNtoJ}
Let the assumptions in Theorem \ref{mainthm} hold, then for any given $\widetilde{f}_{N}\in {\mathbb{X}}$, it holds
\begin{equation*}
\lim_{N\to\infty}\big(J_{N}(\mu_N[\widetilde{f}_N],\widetilde{f}_{N})-J(\rho[\widetilde{f}_N],\widetilde{f}_{N})\big)=0.
\end{equation*}
\end{lemma}

\begin{proof}
According to the definitions of $J_{N}$ and $J$, we have
\begin{equation*}
\begin{split}
&\Big|J_{N}(\mu_N[\widetilde{f}_N],\widetilde{f}_{N})-J(\rho[\widetilde{f}_N],\widetilde{f}_{N})\Big|\\
\leq&\int^T_0\Big|\mathbb{E}(\|j_{\varepsilon}\ast\mu_{N}[\widetilde{f}_N]-z\|_{L^p(\mathbb{R}^d)})-\|\rho[\widetilde{f}_{N}]-z\|_{L^p(\mathbb{R}^d)}\Big| \mathrm{d}t\\
&+\Big|\int_{0}^{T}\int_{\mathbb{R}^d}\widetilde{f}_{N}\rho^{N,\varepsilon; 1}[\widetilde{f}_{N}]\,  \mathrm{d}x\,  \mathrm{d}t-\int_{0}^{T}\int_{\mathbb{R}^d}\widetilde{f}_N\rho[\widetilde{f}_{N}]\,  \mathrm{d}x\,  \mathrm{d}t\Big|\\
:=&A_{1}+A_{2}.
\end{split}
\end{equation*}
The nonlinear term $A_{2}\rightarrow 0$ can be proved by a direct application of the $L^1$ strong convergence result in Proposition \ref{k}, namely
\begin{align*}
A_{2}\leq&\int_{0}^{T}\int_{\mathbb{R}^d}\Big|\widetilde{f}_{N}\rho^{N,\varepsilon; 1}[\widetilde{f}_{N}]-\widetilde{f}_{N}\rho[\widetilde{f}_{N}]\Big|\,  \mathrm{d}x\,  \mathrm{d}t\\
\leq& \|\widetilde{f}_{N}\|_{L^{1}(0,T;L^{\infty}(\mathbb{R}^d))}\|\rho^{N,\varepsilon; 1}[\widetilde{f}_{N}]-\rho[\widetilde{f}_{N}]\|_{L^\infty(0,T;L^1(\mathbb{R}^d))}\leq C(T)\|l\|_{L^1(0,T)}N^{-\beta}\to 0.
\end{align*}
To show the convergence of $A_1$, we have to carefully split it into several terms 
\begin{align*}
		A_{1}\leq&\int^T_0\mathbb{E}\Big(\|j_{\varepsilon}\ast\mu_{N}[\widetilde{f}_{N}]-j_{\varepsilon}\ast\overline{\mu}_{N}[\widetilde{f}_{N}]\|_{L^p(\mathbb{R}^d)}\Big) \mathrm{d}t\\
		&+\int^T_0\mathbb{E}\Big(\|j_{\varepsilon}\ast\overline{\mu}_{N}[\widetilde{f}_{N}]-j_\varepsilon\ast\rho^\varepsilon[\widetilde{f}_{N}]\|_{L^p(\mathbb{R}^d)}\Big) \mathrm{d}t\\
		& + \int^T_0\|j_{\varepsilon}\ast\rho^{\varepsilon}[\widetilde{f}_{N}]-\rho^{\varepsilon}[\widetilde{f}_{N}]\|_{L^p(\mathbb{R}^d)} \mathrm{d}t+\int^T_0\|\rho^{\varepsilon}[\widetilde{f}_{N}]-\rho[\widetilde{f}_{N}]\|_{L^p(\mathbb{R}^d)} \mathrm{d}t:=\sum_{i=1}^{4}A_{1i},
\end{align*}  
where $\overline{\mu}_{N}[\widetilde{f}_{N}]$ denotes the empirical measure of $\overline{X}^\varepsilon_i[\widetilde{f}_{N}]$, $i=1,\cdots,N$. These terms have different structures, $A_{11}$ has the mean-field limit structure, $A_{12}$ has the law of large number structure, $A_{13}$ is simply the convolution approximation, and $A_{14}$ is the approximation of PDE solution. Therefore we handle them separately. By Lemma \ref{e}, \ref{b10} and Proposition \ref{bca}, we have
\begin{align*}
A_{13}\leq &2\int^T_0\Big\|j_{\varepsilon}\ast\rho^{\varepsilon}[\widetilde{f}_{N}]-\rho^{\varepsilon}[\widetilde{f}_{N}]\Big\|^{\frac{p-1}{p}}_{L^\infty(\mathbb{R}^d)} \mathrm{d}t\\
\leq & 2T^{\frac{1}{p}}
\int_{0}^{T}\esssup_{x\in\mathbb{R}^d}\Big|\int_{\mathbb{R}^d}j_{\varepsilon}(x-y)\rho^{\varepsilon}[\widetilde{f}_{N}](y)\,  \mathrm{d}y-\rho^{\varepsilon}[\widetilde{f}_{N}](x)\Big|\,  \mathrm{d}t\\
\leq & 2T^{\frac{1}{p}}
\int_{0}^{T}\esssup_{x\in\mathbb{R}^d}\int_{\mathbb{R}^d}j_{\varepsilon}(x-y)\Big|\rho^{\varepsilon}[\widetilde{f}_{N}](y)\, -\rho^{\varepsilon}[\widetilde{f}_{N}](x)\Big| \mathrm{d}y\,  \mathrm{d}t\\
\leq & 2T^{\frac{1}{p}}
\int_{0}^{T}\esssup_{x\in\mathbb{R}^d}\int_{\mathbb{R}^d}j_{\varepsilon}(x-y)\Big|\rho^{\varepsilon}[\widetilde{f}_{N}](y)\, -\rho^{\varepsilon}[\widetilde{f}_{N}](x)\Big|\frac{|x-y|^{1-\frac{d}{q}}}{|x-y|^{1-\frac{d}{q}}} \mathrm{d}y\,  \mathrm{d}t\\
\leq&2\varepsilon^{1-\frac{d}{q}}\|\rho^{\varepsilon}[\widetilde{f}_{N}]\|_{L^\infty(0,T;C^{1-\frac{d}{q}}(\mathbb{R}^d))}\|j_{\varepsilon}\|_{L^{\infty}(0, T; L^1(\mathbb{R}^d))}T^{1+\frac{1}{p}}\\
\leq & 2\varepsilon^{1-\frac{d}{q}}\|\rho^{\varepsilon}[\widetilde{f}_{N}]\|_{L^\infty(0,T;W^{1, q}(\mathbb{R}^d))}T^{1+\frac{1}{p}}\leq C N^{-\beta(1-\frac{d}{q})}\rightarrow 0.
\end{align*}
The last term $A_{14}$ is directly given by the PDE solution theorem in Lemma \ref{e},
\begin{equation*}
A_{14}=\int^T_0\|\rho^{\varepsilon}[\widetilde{f}_{N}]-\rho[\widetilde{f}_{N}]\|_{L^p(\mathbb{R}^d)} \mathrm{d}t\leq\|\rho^{\varepsilon}[\widetilde{f}_{N}]-\rho[\widetilde{f}_{N}]\|_{L^\infty(0,T;L^1(\mathbb{R}^d)\cap L^\infty(\mathbb{R}^d))}T\leq C\varepsilon\rightarrow 0.
\end{equation*}
For $A_{11}$ we utilize the same idea as in proving the strong $L^1$ convergence in Proposition \ref{k}, namely we decompose $\Omega=\mathcal{A}_\alpha\cup \mathcal{A}_\alpha^c$. By taking $\gamma=\alpha-\beta$ and notice that $\beta<\beta_*$, we have
\begin{align*}
	A_{11}\leq &T^\frac{1}{p}\int^T_0\Big|\mathbb{E}\Big(\Big\|j_{\varepsilon}\ast\mu_{N}[\widetilde{f}_{N}]-j_{\varepsilon}\ast\overline{\mu}_{N}[\widetilde{f}_{N}]\Big\|^\frac{p-1}{p}_{L^\infty(\mathbb{R}^d)}\Big)\Big| \mathrm{d}t\\
	\leq&C(T)\int_{0}^{T}\mathbb{E}\Big((\mathbb{I}_{\mathcal{A}_{\alpha}}+\mathbb{I}_{\mathcal{A}_{\alpha}^c})\Big\|\frac{1}{N}\sum_{i=1}^{N}(j_{\varepsilon}(\cdot-X_{i}^{N, \varepsilon}[\widetilde{f}_{N}])-j_{\varepsilon}(\cdot-\overline{X}_{i}^{\varepsilon}[\widetilde{f}_{N}]))\Big\|_{L^{\infty}}^{\frac{p-1}{p}}\Big) \mathrm{d}t\\
	\leq & C(T)\Big(\big\|\nabla j_{\varepsilon}\big\|_{L^{\infty}(\mathbb{R}^d)}^{\frac{p-1}{p}}N^{-\frac{(p-1)\alpha}{p}}+\|j_{\varepsilon}\|_{L^{\infty}(\mathbb{R}^d)}^\frac{p-1}{p}N^{-\frac{(p-1)\gamma}{p}}\Big)\\
	\leq&C(T)\Big(N^{\frac{p-1}{p}((d+1)\beta-\alpha)}+N^{\frac{p-1}{p}(d\beta-\gamma)}\Big)\leq C(T)N^{\frac{p-1}{p}((d+1)\beta-\alpha)}\rightarrow 0.
\end{align*} 
We do first interpolation inequality for $A_{12}$ and then use the principle of the law of large number. Namely 
for $p\geq 2$,
$$
\|g\|_{L^p(\mathbb{R}^d)}\leq \|g\|_{L^{2m}(\mathbb{R}^d)}^\eta \|g\|_{L^1(\mathbb{R}^d)}^{1-\eta}, \quad m=\lceil \frac{p}{2}\rceil,\, \eta=\frac{2m(p-1)}{(2m-1)p}
$$
and 
$$
\|j_{\varepsilon}\ast\overline{\mu}_{N}[\widetilde{f}_{N}]-j_\varepsilon\ast\rho^\varepsilon[\widetilde{f}_{N}]\|_{L^1(\mathbb{R}^d)}\leq 2, \quad \forall \omega\in\Omega,
$$
we have by using H\"{o}lder's inequality,
\begin{align*}
	A_{12}	\leq & 2\int^T_0\mathbb{E}\Big(\Big\|j_{\varepsilon}\ast\overline{\mu}_{N}[\widetilde{f}_{N}]-j_\varepsilon\ast\rho^\varepsilon[\widetilde{f}_{N}]\Big\|^\frac{2m(p-1)}{(2m-1)p}_{L^{2m}(\mathbb{R}^d)}\Big) \mathrm{d}t\\ 
	\leq &\int^T_0\bigg(\int_{\mathbb{R}^d}\mathbb{E}\Big(\Big|j_{\varepsilon}\ast\overline{\mu}_{N}[\widetilde{f}_{N}](x)-j_\varepsilon\ast\rho^\varepsilon[\widetilde{f}_{N}](t, x)\Big|^{2m}\Big) \mathrm{d}x\bigg)^{\frac{p-1}{(2m-1)p}} \mathrm{d}t.
\end{align*} 
Now we use the strategy of the law of large number and with notation 
$$
\tilde h(t,x,\overline{X}^\varepsilon_i) =j_\varepsilon(x-\overline{X}^\varepsilon_i[\widetilde{f}_{N}])-j_\varepsilon\ast \rho^\varepsilon[\widetilde{f}_{N}](t,x),
$$
we obtain
\begin{align*}
A_{12}	\leq & \int^T_0\bigg(\int_{\mathbb{R}^d}\mathbb{E}\Big(\Big|\frac{1}{N}\sum^N_{i=1}\tilde h(t,x,\overline{X}^\varepsilon_i)\Big|^{2m}\Big) \mathrm{d}x\bigg)^{\frac{p-1}{(2m-1)p}} \mathrm{d}t\\
= & \int^T_0\bigg(\frac{1}{N^{2m}}\int_{\mathbb{R}^d}\mathbb{E}\Big(\Big|\sum^N_{i=1}\tilde h(t,x,\overline{X}^\varepsilon_i)\Big|^{2m}\Big) \mathrm{d}x\bigg)^{\frac{p-1}{(2m-1)p}} \mathrm{d}t.
\end{align*} 
Notice that $\forall x\in\mathbb{R}^d$, it holds
\begin{align*}
\mathbb{E}(\tilde{h}(t,x,\overline{X}^\varepsilon_i))=\int_{\mathbb{R}^d}\tilde{h}(t,x,y)\rho^\varepsilon[\widetilde{f}_N](t, y) \mathrm{d}y=0.
\end{align*}
In all the terms in the summation, if $\tilde{h}(t,x,\overline{X}^\varepsilon_i)$ appears only once, due to the fact that $\overline{X}^\varepsilon_j$, $j=1,\cdots, N$, are independent, we have
\begin{align*}
&\int_{\mathbb{R}^d}\mathbb{E}\Big(\tilde h(t,x,\overline{X}^\varepsilon_i)\prod^{2m-1}_{n=1, j_n\neq i}\tilde h(t,x,\overline{X}^\varepsilon_{j_n})\Big) \mathrm{d}x\\
=&\int_{\mathbb{R}^d}\mathbb{E}\Big(\tilde h(t,x,\overline{X}^\varepsilon_i)\Big)\mathbb{E}\Big(\prod^{2m-1}_{n=1, j_n\neq i}\tilde h(t,x,\overline{X}^\varepsilon_{j_n})\Big) \mathrm{d}x =0.
\end{align*}
By directly combination arguments, the number of terms where all the $\tilde{h}(t,x,\overline{X}^\varepsilon_i)$'s appear at least twice can be bounded by $C(m)N^m$. In addition, all these terms can be estimated by
\begin{align*}
&\int_{\mathbb{R}^d}\mathbb{E}\Big(\prod^{2m}_{n=1, j_n}\tilde h(t,x,\overline{X}^\varepsilon_{j_n})\Big) \mathrm{d}x\\
\leq &\sup_{(t,x,y)\in (0,T)\times\mathbb{R}^{2d}}|\tilde{h}(t,x,y)|^{2m-1}\int_{\mathbb{R}^d}\mathbb{E}\Big(\Big|\tilde h(t,x,\overline{X}^\varepsilon_{j_1})\Big|\Big) \mathrm{d}x\\
\leq &\|j_\varepsilon\|_{L^\infty(\mathbb{R}^d)}^{2m-1}\int_{\mathbb{R}^d}\int_{\mathbb{R}^d}\Big|j_\varepsilon(x-y)-j_\varepsilon\ast\rho^\varepsilon[\widetilde{f}_N](t,x)\Big|\rho^\varepsilon[\widetilde{f}_N](t,y) \mathrm{d}y \mathrm{d}x
\leq 2N^{\beta d(2m-1)}.
\end{align*}
Therefore the estimate for $A_{12}$ by the above discussion can be summarized into, for $\beta<\beta^*\leq \frac{1}{2d}$,
\begin{align*}
A_{12}	\leq & C(m)T\bigg(\frac{N^{\beta d(2m-1)}}{N^{m}}\bigg)^{\frac{p-1}{(2m-1)p}}\leq C(p)TN^{-\frac{\beta d(p-1)}{(2m-1)p}}\rightarrow 0.
\end{align*}
Hence we obtain the result 
\begin{equation*}
\Big|(J_{N}(\mu_N[\widetilde{f}_N],\widetilde{f}_{N})-J(\rho[\widetilde{f}_N],\widetilde{f}_{N}))\Big|\rightarrow 0, \quad \mbox{as } N\rightarrow \infty.
\end{equation*}
\end{proof}

\begin{lemma}\label{lemweakcompact}
For a sequence $f_{N}\in {\mathbb{X}}$, there exists a convergent subsequence $f_{N_k}\rightharpoonup f$ in ${\mathbb{X}}$ such that the corresponding solution $\rho[f_{N_k}]$ of \eqref{16} converges to $\rho[f]$ in the following sense
\begin{equation*}
\rho[f_{N_{k}}]\rightharpoonup\rho[f]\ \text{in}\ L^{1}(0, T; L^p(\mathbb{R}^d))
\end{equation*}
and
\begin{equation*}
\rho[f_{N_{k}}]\to\rho[f]\ \text{in}\ L^2(0, T; L^1(\mathbb{R}^d)).
\end{equation*}
\end{lemma}

We omit the proof of this lemma, since the estimates are similar to those shown in Section \ref{Nproblem}, with which one can easily get the compactness by applying Aubin-Lions lemma.

\vskip3mm
Now we are ready to finish the proof of Theorem \ref{mainthm}.
\begin{proof}
For any fixed $N$, by Proposition \ref{propNcontrol} we know that there exists a minimizer of $J_{N}(\mu_N[f],f)$. Without loss of generality, we denote $f_{N}$ to be a minimizer, i.e. 
$$
J_{N}(\mu_N[f_N], f_N)=\min_{f\in {\mathbb{X}}}J_N(\mu_N[f],f).
$$
Due to the weak compactness of ${\mathbb{X}}$ and the compactness result in Lemma \ref{lemweakcompact}, there exists a subsequence of $f_{N_k}\in {\mathbb{X}}$ such that 
\begin{align}
\label{weakconvergencef}&f_{N_{k}}\rightharpoonup\overline{f} \qquad\qquad\mbox{ in } L^r(0, T; W^{1, q}(\mathbb{R}^d)),\\
\label{weakconvergencerho}& \rho[f_{N_k}] \rightharpoonup \rho[\overline{f}] \qquad\mbox{ in } L^{1}(0, T; L^p(\mathbb{R}^d)),\\
\label{strongconvergencerho}&\rho[f_{N_k}] \rightarrow \rho[\overline{f}] \qquad\mbox{ in } L^2(0, T; L^1(\mathbb{R}^d)).
\end{align}

In the following, we show that $\displaystyle\lim_{N\to\infty} J_{N}(\mu_N[f_N],f_{N})=J(\rho[\overline{f}],\overline{f})=\min_{f\in {\mathbb{X}}} J(\rho[f],f)$. 
We will use the $\Gamma$-convergence strategy, where the following two limits are important
\begin{equation} \label{step1}
\liminf_{k\to\infty} J_{N_{k}}(\mu_{N_k}[f_{N_k}], f_{N_{k}})\geq J(\rho[\overline{f}], \overline{f})
\end{equation}
and
\begin{equation}\label{step2}
\lim_{N\to\infty} J_{N}(\mu_N[\overline{f}], \overline{f})=J(\rho[\overline{f}], \overline{f}).
\end{equation}
The limit \eqref{step2} can be directly obtained by taking $\widetilde{f}_N=\overline{f}$ in Lemma \ref{lemmaJNtoJ}.

If \eqref{step1} is valid, then we have
\begin{align*}
\limsup_{N\to\infty}\min_{f\in {\mathbb{X}}} J_{N}(\mu_N[f], f)\leq&\lim_{N\to\infty} J_{N}(\mu_N[\overline{f}], \overline{f})
=J(\rho[\overline{f}], \overline{f})\\
\leq &\liminf_{k\to\infty}J_{N_{k}}(\mu_{N_k}[f_{N_k}], f_{N_{k}})\leq\liminf_{k\to\infty}\min_{f\in {\mathbb{X}}}J_{N_{k}}(\mu_{N_k}[f],f).
\end{align*}
Therefore, we obtain immediately  
$$
\displaystyle\lim_{k\to\infty}\min_{f\in {\mathbb{X}}}J_{N_{k}}(\mu_{N_{k}}[f],f)=J(\rho[\overline{f}], \overline{f}).
$$ 
As $f_{N_{k}}$ is the minimizer of $J_{N_{k}}(\mu_{N_{k}}[f], f)$, then we have $\displaystyle\lim_{k\to\infty} J_{N_{k}}(\mu_{N_{k}}[f_{N_{k}}], f_{{N_{k}}})=J(\rho[\overline{f}],\overline{f})$. 

Now we are left to prove that $\overline{f}$ is a minimizer of $J(\rho[f], f)$.  Obviously,  $J(\rho[\overline{f}], \overline{f})\geq \displaystyle\min_{f\in {\mathbb{X}}}J(\rho[f], f)$, which means that we only need to prove $J(\rho[\overline{f}], \overline{f})\leq\displaystyle\min_{f\in {\mathbb{X}}}J(\rho[f], f)$. Actually, for any $\hat{f}\in {\mathbb{X}}$, by taking $\widetilde{f}_N=\hat{f}$ in Lemma \ref{lemmaJNtoJ}, we have 
\begin{equation*}
J(\rho[\hat{f}], \hat{f})=\lim_{N\to\infty}J_{N}(\mu_N[\hat{f}], \hat{f})\geq\lim_{N\to\infty}\min_{f\in {\mathbb{X}}}J_{N}(\mu_N[f], f)=J(\rho[\overline{f}], \overline{f}).
\end{equation*} 

In the end, we prove the lower semi-continuity property in \eqref{step1} to finalize the whole argument. The proof of \eqref{step1} can be reduced to prove that 
\begin{align*}
\lim_{k\to\infty} \Big| J_{N_{k}}(\mu_{N_k}[f_{N_k}], f_{N_{k}})- J(\rho[f_{N_k}],f_{N_k})\Big| + \liminf_{k\to\infty} J(\rho[f_{N_k}], f_{N_{k}})\geq J(\rho[\overline{f}],\overline{f}).
\end{align*}
Actually, by applying Lemma \ref{lemmaJNtoJ} with $\widetilde{f}_{N_k}=f_{N_k}$ the limit of the first term vanishes, i.e.
\begin{align*}
\lim_{k\to\infty} \Big| J_{N_{k}}(\mu_{N_k}[f_{N_k}], f_{N_{k}})- J(\rho[f_{N_k}],f_{N_k})\Big|=0.
\end{align*}
Then the proof is finished if we can prove
\begin{align*}
\liminf_{k\to\infty} J(\rho[f_{N_k}], f_{N_{k}})\geq J(\rho[\overline{f}],\overline{f}).
\end{align*}
The weak convergence of $\rho[f_{N_k}]$ in \eqref{weakconvergencerho} implies
\begin{align*}
\liminf_{k\to\infty}\int^T_0\|\rho[f_{N_k}]-z\|_{L^p(\mathbb{R}^d)}\mathrm{d}t \geq \int^T_0\|\rho[\overline{f}]-z\|_{L^p(\mathbb{R}^d)}\mathrm{d}t.
\end{align*}
From the strong convergence of $\rho[f_{N_{k}}]$ in \eqref{strongconvergencerho} and weak convergence of $f_{N_{k}}$ in \eqref{weakconvergencef}, we obtain the convergence of the second term in the cost functional, i.e.
\begin{equation*}
\int_{0}^{T}\int_{\mathbb{R}^d}f_{N_{k}}\rho[f_{N_{k}}]\,  \mathrm{d}x\,  \mathrm{d}t\to\int_{0}^{T}\int_{\mathbb{R}^d}\overline{f}\rho[\overline{f}]\,  \mathrm{d}x\,  \mathrm{d}t,
\end{equation*}
which yields the proof of this theorem.
\end{proof}

\section*{Appendix}
In the appendix, we give a list of lemmas which have been constantly used in this paper. Notice that we use the notation $\Phi$, $\widetilde{\Phi}$, and $\widetilde{\Phi}_\varepsilon$ from \eqref{Phis}. 
\begin{lemma}\label{f}
	For any $g\in L^\infty(0, T; W^{1, q}(\mathbb{R}^d))$,  we have
	\begin{equation*}
		\|\nabla(\Phi_{\varepsilon}-\Phi)\ast g\|_{L^{\infty}(\mathbb{R}^d)}\leq\varepsilon\|g\|_{W^{1,q}(\mathbb{R}^d)}.
	\end{equation*}
\end{lemma}

\begin{proof}
	According to the properties of convolution, we have
	\begin{equation}\label{20}
		\nabla(\Phi_{\varepsilon}-\Phi)\ast g=j_{\varepsilon}\ast(\nabla\Phi\ast g)-\nabla\Phi\ast g.
	\end{equation}
	Let $h=\nabla\Phi\ast g$, then by using the mean value theorem, we have
	\begin{equation}
		\begin{split}
			j_{\varepsilon}\ast h-h=&\int_{\mathbb{R}^d} j_{\varepsilon}(x-y)(h(y)-h(x))\,  \mathrm{d}y\\
			\leq&\int_{\mathbb{R}^d}j_{\varepsilon}(x-y)\nabla h(\xi)\cdot(y-x)\,  \mathrm{d}y\leq \varepsilon\|\nabla h\|_{L^{\infty}(\mathbb{R}^d)},
		\end{split}
	\end{equation}
where $\xi\in(x, y)$. Therefore by Sobolev embedding theorem we obtain
	\begin{equation*}
		\|\nabla(\Phi_{\varepsilon}-\Phi)\ast g\|_{L^{\infty}(\mathbb{R}^d)}\leq\varepsilon\|\nabla^2\Phi\ast g\|_{L^{\infty}(\mathbb{R}^d)}\leq\varepsilon\|\nabla^2\Phi\ast g\|_{W^{1, q}(\mathbb{R}^d)}\leq\varepsilon\|g\|_{W^{1, q}(\mathbb{R}^d)}.
	\end{equation*}
\end{proof}

\begin{lemma}\label{phi1}
For any $g\in L^1(\mathbb{R}^d)\cap L^\infty(\mathbb{R}^d)$, we have
	\begin{equation*}
		\|\nabla{\Phi}\ast g\|_{L^{\infty}(\mathbb{R}^d)}+\|\nabla\widetilde{\Phi}\ast g\|_{L^{\infty}(\mathbb{R}^d)}\leq C(d)(\|g\|_{L^1(\mathbb{R}^d)}+\|g\|_{L^{\infty}(\mathbb{R}^d)}).
	\end{equation*}
\end{lemma}

\begin{proof}
	According to the definition of $\Phi$ and $\widetilde{\Phi}$ in \eqref{Phis}, we have
	\begin{align*}
				\nabla{\Phi}\ast g&=C_{d}(d-2)\Big(\int_{|x-y|\leq 1}\frac{x-y}{|x-y|^d}g(y)\,  \mathrm{d}y+\int_{|x-y|>1}\frac{x-y}{|x-y|^d}g(y)\,  \mathrm{d}y\Big),\\
		\nabla\widetilde{\Phi}\ast g=&C_{d}(d-2)\Big((2\varepsilon)^{1-d}\int_{|x-y|\leq 2\varepsilon}\frac{x-y}{|x-y|}g(y)\, \mathrm{d}y+\int_{2\varepsilon<|x-y|\leq 1}\frac{x-y}{|x-y|^d}g(y)\,  \mathrm{d}y\Big)\\
		&+C_{d}(d-2)\int_{|x-y|>1}\frac{x-y}{|x-y|^d}g(y)\,  \mathrm{d}y.
	\end{align*}
It is easy to get
	\begin{equation*}
	\|\nabla{\Phi}\ast g\|_{L^{\infty}(\mathbb{R}^d)}+\|\nabla\widetilde{\Phi}\ast g\|_{L^{\infty}(\mathbb{R}^d)}\leq C(d)(\|g(y)\|_{L^{\infty}(\mathbb{R}^d)}+\|g(y)\|_{L^1(\mathbb{R}^d)}).
	\end{equation*}
\end{proof}

\begin{lemma}\label{g}
	For any $g\in L^\infty(\mathbb{R}^d)$, we have
	\begin{equation*}
		\|\nabla(\Phi-\widetilde{\Phi})\ast g\|_{L^{\infty}(\mathbb{R}^d)}\leq\varepsilon C(d)\|g\|_{L^{\infty}(\mathbb{R}^d)}.
	\end{equation*}
\end{lemma}

\begin{proof}
	According to the definition of $\Phi$ and $\widetilde{\Phi}$, we have
	\begin{equation*}
		\nabla(\Phi-\widetilde{\Phi})=
		\begin{cases}
			0, &|x|\geq 2\varepsilon,\\
			-C_{d}(2-d)\frac{x_{i}}{|x|}(\frac{1}{|x|^{d-1}}-\frac{1}{(2\varepsilon)^{d-1}}), &|x|< 2\varepsilon,
		\end{cases}
	\end{equation*}
	which yields that
	\begin{equation*}
		\nabla(\Phi-\widetilde{\Phi})\ast g=-C_{d}(2-d)\Big(\int_{|x-y|<2\varepsilon}\frac{x-y}{|x-y|^{d}}g(y)\,  \mathrm{d}y-(2\varepsilon)^{1-d}\int_{|x-y|<2\varepsilon}\frac{x-y}{|x-y|}g(y)\, \mathrm{d}y\Big).
	\end{equation*}
	Then the $L^{\infty}$-norm will be obtained by
	\begin{equation*}
	\|\nabla(\Phi-\widetilde{\Phi})\ast g\|_{L^{\infty}(\mathbb{R}^d)}\leq\varepsilon C(d)\|g\|_{L^{\infty}(\mathbb{R}^d)}.
	\end{equation*}
\end{proof}

\begin{lemma}\label{bcyl1}
For any function $f\in L^{1}(\mathbb{R}^d)\cap W^{1, q}(\mathbb{R}^{d})$, we have
\begin{equation*}
\|\nabla^2\widetilde{\Phi}\ast f\|_{L^{\infty}(\mathbb{R}^d)}\leq C(d)(\|f\|_{L^1(\mathbb{R}^d)}+\|f\|_{W^{1, q}(\mathbb{R}^d)}).
\end{equation*}
\end{lemma}

\begin{proof}
According to the definition of convolution, we know that
\begin{equation}\label{BC100}
\begin{split}
&\nabla^2\widetilde{\Phi}\ast f=\nabla\widetilde{\Phi}\ast\nabla f\\
=&-C_{d}(2-d)(2\varepsilon)^{1-d}\int_{|x-y|<2\varepsilon}\frac{x-y}{|x-y|}\nabla f(y)\, \mathrm{d}y\\
&-C_{d}(2-d)\int_{2\varepsilon<|x-y|\leq 1}\frac{x-y}{|x-y|^{d}}\nabla f(y)\, \mathrm{d}y-C_{d}(2-d)\int_{|x-y|>1}\frac{x-y}{|x-y|^{d}}\nabla f(y)\, \mathrm{d}y\\
:=& Q_{1}+Q_{2}+Q_{3}.
\end{split}
\end{equation} 
For the nonlinear term $Q_{1}$ and $Q_{2}$, we have
\begin{equation*}
\|Q_{1}\|_{L^{\infty}(\mathbb{R}^d)}\leq C(d)(2\varepsilon)^{1-d}(2\varepsilon)^{d-\frac{1}{q}}\|\nabla f\|_{L^q(\mathbb{R}^d)}\leq C(d)(2\varepsilon)^{1-\frac{1}{q}}\|\nabla f\|_{L^q(\mathbb{R}^d)}
\end{equation*}
and
\begin{equation*}
\|Q_{2}\|_{L^{\infty}(\mathbb{R}^d)}\leq C(d)\|\nabla f\|_{L^q(\mathbb{R}^d)}.
\end{equation*}
For the nonlinear term $Q_{3}$, by the integration by parts, we have
\begin{equation*}
\begin{split}
Q_{3}=&C_{d}(2-d)\int_{|x-y|>1}\frac{|x-y|^d-d|x-y|^{d-2}(x_{i}-y_{i})(x_{j}-y_{j})}{|x-y|^{2d}}f(y)\, \mathrm{d}y,\\
&-C_{d}(2-d)\int_{|x-y|=1}\frac{x-y}{|x-y|^d}f(y)\nu\, \mathrm{d}S,
\end{split}
\end{equation*}
where $\nu$ is the outward pointing unit normal vector. Then we have
\begin{equation*}
\|Q_{3}\|_{L^{\infty}(\mathbb{R}^d)}\leq C(d)\|f\|_{L^1(\mathbb{R}^d)}.
\end{equation*}
Putting all the estimates of $Q_{i}$ into (\ref{BC100}), we have
\begin{equation*}
\|\nabla^2\widetilde{\Phi}\ast f\|_{L^{\infty}(\mathbb{R}^d)}\leq C(d)(\|f\|_{L^1(\mathbb{R}^d)}+\|\nabla f\|_{L^q(\mathbb{R}^d)}).
\end{equation*}
\end{proof}

\begin{lemma}\label{compactembedding}
For any $\kappa\in[1, \frac{2d}{d-2})$, we have the following compact embedding
\begin{equation*}
H^1(\mathbb{R}^d)\cap L^1(\mathbb{R}^d;(1+|x|^2)\mathrm{d}x)\hookrightarrow\hookrightarrow L^{\kappa}(\mathbb{R}^d).
\end{equation*}
\end{lemma}

\begin{proof}
For any bounded sequence $g_{n}$ in $H^1(\mathbb{R}^d)$, we have
that for any $R>0$, $g_{n}$ has a convergence subsequence $g_{n_R}$ which converges strongly to a function $g\in L^1(B_R)$, namely
$$
\|g_{n_R}-g\|_{L^1(B_R)}\rightarrow 0,\quad \mbox{ as } n_R\rightarrow \infty.
$$

For any fixed $R_0$, we have for $R_0<R_j\nearrow\infty$, by Cantor's diagonal method, we can take a common subsequence $g_{n_j}$ which convergence to $g$ in the following sense
$$
\|g_{n_j}-g\|_{L^1(B_{R_0})}\rightarrow 0,\quad \mbox{ as } j\rightarrow \infty.
$$

Then for any small $\eta>0$, there exists $R_\eta$ such that 
$$
\frac{4M_2}{R^2_\eta}\leq\frac{\eta}{2}, \qquad \mbox{where } \displaystyle\int_{\mathbb{R}^d}|x|^2|g_n(x)|\mathrm{d}x\leq M_2.
$$
Furthermore, there exists $N>0$ such that for $n_j>N$ it holds
$$
\|g_{n_j}-g\|_{L^1(B_{R_\eta})}\leq \frac{\eta}{2}.
$$
As a summary, for $n_j>N$ we have
\begin{equation*}
\int_{\mathbb{R}^d} |g_{n_j}-g|\,  \mathrm{d}x\leq\int_{B_{R_\eta}} |g_{n_j}-g|\,  \mathrm{d}x+\int_{B_{R_\eta}^c}\frac{|x|^2}{R_\eta^2}|g_{n_j}-g|\,  \mathrm{d}x\leq\frac{\eta}{2}+\frac{4M_2}{R_\eta^2}\leq\eta,
\end{equation*}
which means $\{g_{n}\}$ is compact in $L^1(\mathbb{R}^d)$. For all $\kappa\in(1, \frac{2d}{d-2})$, using Gagliardo-Nirenberg's inequality, we have
\begin{equation*}
\|g_{n_j}-g\|_{L^\kappa(\mathbb{R}^d)}\leq\|g_{n_j}-g\|_{L^1(\mathbb{R}^d)}^{\frac{2(d+\kappa)-d\kappa}{(d+2)\kappa}}\|\nabla(g_{n_j}-g)\|_{L^2(\mathbb{R}^d)}^{\frac{2d(\kappa-1)}{(d+2)\kappa}}.
\end{equation*}
\end{proof}

\section*{Acknowledgments}

The research of Li Chen was supported in part by the German Research Foundation (No. CH 955/8-1). The research of Yucheng Wang was supported in part by National Natural Science Foundation of China (No.12271357), Shanghai Science and Technology Innovation Action Plan (No.21JC1403600), and the CSC-DAAD Postdoc Scholarship. Yucheng Wang thanks the School of Business Informatics and Mathematics, University of Mannheim for kindly host.

\end{document}